\documentclass{article}

\newcommand{\xx}{{\sf X}}
\newcommand{\yy}{{\sf Y}}
\newcommand{\zz}{{\sf Z}}
\newcommand{\uu}{{\sf U}}
\newcommand{\rr}{\widetilde{\sf R}}

\newcommand{\idel}[1]{{\ensuremath {\mathrm{I}\Delta_{#1}}}\xspace}

\newcommand{\xxp}{\overline{\sf X}}
\newcommand{\yyp}{\overline{\sf Y}}
\newcommand{\zzp}{\overline{\sf Z}}
\newcommand{\uup}{\overline{\sf U}}

\usepackage{amsthm,bm, amssymb, stmaryrd}
\usepackage{proof}	
\usepackage{color}
\usepackage{amsmath}
\usepackage{gastex}

\newcommand{\la}{\langle}
\newcommand{\ra}{\rangle}

\newcommand{\principle}[1]{\ensuremath{\formal{#1}}}

\newcommand{\formal}[1]{\ensuremath{{\sf {#1}}}\xspace}
\newcommand{\sonetwo}{{\ensuremath{{\sf S^1_2}}}\xspace}

\newcommand{\il}{{\ensuremath{\textup{\textbf{IL}}}}\xspace}
\newcommand{\extil}[1]{\ensuremath{\textup{\textbf{IL}}{\sf\ensuremath{#1}}}\xspace}
\newcommand{\gl}{{\ensuremath{\textup{\textbf{GL}}}}\xspace}

\newcommand{\intl}[1]{{\ensuremath {\textup{\textbf{IL}}}({\rm #1})}}

\newcommand{\ilm}{\extil{M}}
\newcommand{\ilp}{\extil{P}}
\newcommand{\ilw}{\extil{W}}

\newcommand{\ilwstar}{\extil{W^*}}

\newcommand{\ilal}{\intl{All}\xspace}

\newcommand{\cut}[1]{{\sf Cut}(#1)}
\newcommand{\forallcut}{\forall^{\sf Cut}}
\newcommand{\existscut}{\exists^{\sf Cut}}

\theoremstyle{plain}
\newtheorem{theorem}{Theorem}[section]
\newtheorem{definition}[theorem]{Definition}
\newtheorem{lemma}[theorem]{Lemma}

\newtheorem{corollary}[theorem]{Corollary}

\theoremstyle{remark}

\theoremstyle{question}

\newtheorem{conjecture}[theorem]{Conjecture}

\usepackage{xspace}
\usepackage{amssymb}

\usepackage{psfig}
\usepackage{epsfig}

\title{Two series of formalized interpretability principles for weak systems of arithmetic}
\author{Evan Goris, Joost J. Joosten}

\begin{document}
\maketitle

\begin{abstract}
The provability logic of a theory $T$ captures the structural behavior of formalized provability in $T$ as provable in $T$ itself. Like provability, one can formalize the notion of relative interpretability giving rise to interpretability logics. Where provability logics are the same for all moderately sound theories of some minimal strength, interpretability logics do show variations. 

The logic \ilal is defined as the collection of modal principles that are provable in any moderately sound theory of some minimal strength. In this paper we raise the previously known lower bound of \ilal by exhibiting two series of principles which are shown to be provable in any such theory. Moreover, we compute the collection of frame conditions for both series. 
\end{abstract}

\section{Introduction}

Relative interpretations in the sense of Tarski, Mostowski and Robinson \cite{TarskiEtAl:1953:UndecidableTheories} are widely used in mathematics and in mathematical logic to interpret one theory into another. Roughly speaking, such an interpretation between two theories is a \emph{translation} from the language of one theory to the language of the other so that the translation preserves logical structure and theoremhood. 

We shall write $U\rhd V$ to denote that a theory $U$ interprets a theory $V$. Once we know that $U\rhd V$, this provides us much information; for example the consistency of $U$ implies the consistency of $V$ and also, various definability results carry over from the one theory to the other. Famous examples of interpretations are abundant: the theory of the natural numbers into the theory of the integers, set theory plus the continuum hypothesis into ordinary set theory, non-Euclidean geometry into Euclidean geometry, etc.

Interpretability, being a syntactical notion, allows for formalization very much as one can formalize the notion of provability. As such, we can consider \emph{interpretability logics} which actually extend the well-know provability logic \gl of G\"odel L\"ob. We shall see that the interpretability logic of a theory is the collection of all structural properties of interpretability that it can prove. 

Where all modestly correct theories of some minimal strength --let us call them \emph{reasonable theories} in this paper-- have the same provability logic \gl, the situation is different in the case of interpretability and different theories have different logics. It is an open question to determine the logic of interpretability principles being provable in any reasonable theory. This paper reports on substantial progress on this open question by increasing the previously known lower bound.

\section{Preliminaries}
Let $U$ and $V$ denote theories with languages $\mathcal L_U$ and $\mathcal L_V$ respectively. A relative interpretation $j$ from $V$ into a theory $U$ --we will write $j:U\rhd V$-- is a pair $\la \delta (x), t\ra$ where $\delta(x)$ is a formula of $\mathcal L_U$ that specifies the domain in which $V$ will be interpreted and $t$ is a translation, mapping symbols of $\mathcal L_V$ to formulas of $\mathcal L_U$ providing a definition in $U$ of these symbols. 

The translation $t$ is extended to a translation $j$ of formulas in the usual way by having $j$ commute with the connectives and relativize the quantifiers to the domain specifier $\delta(x)$ as follows: $\big( \forall x\ \varphi (x) \big)^j := \forall x\ \big( \delta(x) \to \varphi^j (x) \big)$. We will not go too much into details but the main point is that interpretations are primarily syntactical notions --especially for finite languages-- and as such allow for an arithmetization/formalization very much as formal proofs do. 

\subsection{Arithmetic}

In order to formalize the notion of interpretability within some base theory $T$ one needs to require some minimal strength conditions on $T$. In particular, we shall require that $T$ can speak of numbers where to code syntax and without loss of generality we shall assume that $T$ contains the language of arithmetic $\{ +, \times, S, 0, 1 <, = \}$. 

We will need that the main properties of the basic syntactical operations like substitution are provable within $T$. For reasonable coding protocols this implies that we need to require the totality of a function of growth-rate $\omega_1(x) := x \mapsto 2^{2|x|}$ where $|x|$ denotes the integral part of the binary logarithm of $x$. 

Further, to perform basic arguments we need a minimal amount of induction and actually a surprisingly little amount of induction suffices. Buss's theory \sonetwo has just the needed amount of induction and proves the totality of $\omega_1$ and this shall be our base theory (formulated in the standard language of arithmetic).

Alternatively, we could have taken as base theory $\idel{0} + \Omega_1$ which consists of Robinson's arithmetic $\sf Q$ together  with induction for bounded formulas with parameters and the axiom $\Omega_1$ stating that the graph of $\omega_1$ defines a total function. We refer the reader to \cite{HajekPudlak:1993:Metamathematics} and \cite{Buss:1986:BoundedArithmetic} for further details.

A sharply bounded quantifier is one of the form $\forall \, x{<}|y|$ where $|y|$ denotes the integer value of the binary logarithm of $y$. The class $\Delta^b_0$ contains exactly the formulas where each quantifier is sharply bounded. The class $\Sigma^1_b$ arises by allowing bounded existential quantifiers and sharply bounded universal quantifiers to occur over $\Delta^b_0$ formulas. By $\exists \Sigma^1_b$ we denote those formulas that arise by allowing a single unbounded existential quantifier over a $\Sigma^1_b$ formula. The complexity classes $\Pi_n$, $\Sigma_n$ and $\Delta_n$ refer to the usual quantifier alternations hierarchies in the standard language of arithmetic.

In this paper we shall only be concerned with first order-theories in the language of arithmetic with a poly-time recognizable set of axioms extending \sonetwo and shall often refrain from repeating (some of) these conditions. We shall write $\Box_T \phi$ as the $\exists \Sigma_1^b$ formalization of $\phi$ being provable in the theory $T$ and refrain from distinguishing formulas from their G\"odel numbers or even the numerals thereof. It is well known that we can express provable $\Sigma_1$ completeness using formalized provability.

\begin{lemma}
For any theory $T$ extending \sonetwo we have that
\[
T\vdash \forall \alpha \ \Box_T \alpha \to \Box_T \Box_T \alpha.
\]
\end{lemma}

We will use $U\rhd V$ to denote the formalization of ``the theory $V$ is interpretable in the theory $U$". If we abbreviate the existential quantifier over numbers that code a pair $\la \delta (x), t\ra$ defining an interpretation by $\exists^{int}j$ we can write
\begin{equation}\label{equation:definitionOfInterpretabilityFormalized}
U\rhd V := \exists^{int} j\, \forall \psi \ (\Box_V \psi \to \Box_U \psi^j).
\end{equation}

An interpretation $j:U\rhd V$ can be used as a uniform way to obtain a model of $V$ inside any model of $U$. If $U$ satisfies full induction, then we see that actually the defined model of $V$ is an end extension of the model of $U$: we define $f(0):=0$ and $f(x+1) := f(x)+^j 1^j$ and by induction see that $\forall x\exists y \ f(x)=y$. As such, we see that any $\Sigma_1$ consequence of $U$ must necessarily also hold in $V$. Since $\Box_T \varphi$ is a $\Sigma_1$ formula, the insight on end extensions is reflected in what is called \emph{Montagna's principle} 
\begin{equation}\label{equation:MontagnaForTheories}
\ \ \ (U\rhd V) \to \big( (U \cup \{\Box_T \varphi\})\rhd (V \cup \{ \Box_T \varphi\} )\big).
\end{equation}
In case $U$ does not have full induction, we can still define the graph $F(x,y)$ of the function $f$ from above, but we can no longer prove that the function is total. However, we can prove that $\exists y \ F(x,y)$ is \emph{progressive}, that is, we can prove
\[
\exists y \, F(0,y) \ \wedge \ \forall x\ \big( \exists y \, F(x,y) \to \exists y, \ F(x+1,y)\big).
\]
In particular, the formula $\exists y \, F(x,y)$ defines an initial segment within $U$. A common trick in weak arithmetics is to use this initial segment as our natural numbers instead of applying induction (which is not necessarily available). By Solovay's techniques on shortening initial segments we may assume that they obey certain closure properties giving rise to the what is called a \emph{definable cut}.

A formula $J$ is called a $T$-cut whenever $T$ proves all of
\begin{enumerate}
\item
$J(0) \wedge \forall x \, (J(x) \to J(x+1))$;

\item
$\forall x \, \big(J(x) \wedge J(y) \to J(x+y)\wedge J(xy) \wedge J(\omega_1(x))\big)$;

\item
$J(x) \wedge y<x \to J(y)$.
\end{enumerate}
Let $\cut J$ denote the conjunction of these three requirements. Sometimes we want to quantify over cuts within $T$ so that these cuts can then of course be non-standard. We shall use $\forallcut J\ \psi$ and $\existscut J\ \psi$ to denote $\forall J\ (\, \Box_T \cut {\dot J} \to \psi)$ and $\exists J\ (\, \Box_T \cut {\dot J} \wedge \psi)$ respectively. Here the dot notation in $\Box_T \cut {\dot J}$ is the standard way to abbreviate the formula with one free variable $J$ stating that the formula $\cut J$ is provable in $T$. We will freely use the dot notation throughout the remainder of this paper. Sometimes we shall write $x{\in} J$ instead of $J(x)$.

For $J$ a cut, let $\psi^J$ denote the formula where all unrestricted quantifiers are now required to range over values in $J$. That is, $(\forall x \ \phi)^J: = \forall x \ (J(x) \to \phi^J)$, $(\exists x \ \phi)^J: = \exists x \ (J(x) \wedge \phi^J)$. Moreover, that is the only thing that is done by this translation so that for example $(\phi \wedge \psi)^J := \phi^J \wedge \psi^J$ etc. Instead of writing $(\Box_T\phi)^J$ we shall simply write $\Box_T^J \phi$. We note that if $\psi(J)\in \exists \Sigma_1^b$, then $\existscut J\ \psi(J)$ is again provably equivalent to an $\exists \Sigma^b_1$ formula.

Let us get back to the role of induction in Montagna's principle. If $j: U\rhd V$ and $U$ does not prove full induction, then $j$ will not define an end extension of any model of $U$. However, it is easy to see that $j$ does define, using the progressive formula $\exists y\, F(x,y)$, a definable cut in $U$ on which $f$ is an isomorphism. This is reflected in a weakening of Montagna's principle also referred to as \emph{Pudl\'ak's principle}.
\begin{lemma}\label{theorem:PudlaksPrinciple} 
Let $T$ be a theory containing $\sonetwo$ and let $U$ and $V$ be theories.
\begin{equation}\label{equation:PudlakForTheories}
T\vdash U\rhd V \to \existscut J \, \forall \, \psi {\in} \Delta_0  \, \Big( U\cup \{(\exists x \,  \psi)^{\dot J}\} \rhd V\cup \{ \exists x\, \psi \} \Big).
\end{equation}
\end{lemma}

\subsection{The interpretability logic of a theory}

Interpretability logics are designed to capture structural behavior of formalized interpretability just as provability logic captures the structural behavior of formalized provability. To this end we consider a propositional modal language with a unary modal operator $\Box$ to model formalized provability and a binary modal operator $\rhd$ to model formalized interpretability of sentential extensions of some base theory. Let us make this more precise.

Let us fix an arithmetical theory $T$; By $*$ we will denote a \emph{realization}, that is, any mapping from the set of propositional variables to sentences of $T$. The map $*$ is extended to the set of all modal formulas of interpretability logics as follows
\[
\begin{array}{llll}
(\neg A)^* & := & \neg A^* & \\
(A \wedge B)^* & := & A^* \wedge B^* & \mbox{and likewise for the other connectives}\\
(\Box A)^* &:=& \Box_T A^* & \\
(A\rhd B)^* &:=& (T+ A^*)\rhd (T+B^*). & \\
\end{array}
\]
We can now define the interpretability logic of a theory as those modal principles which are provable under any realization. With some liberal notation this is captured in the following.
\begin{definition}
Let $T$ be a theory containing \sonetwo. We define the \emph{interpretability logic of $T$} as 
\[
\intl{T} \ :=\ \{ A \mid \forall * \ T\vdash A^* \}.
\]
Further, we define the interpretability logic of all arithmetical theories extending \sonetwo by 
\[
\intl{All} \ :=\ \{ A \mid \forall T\, \forall * \ T\vdash A^* \}.
\]

\end{definition}

As a direct corollary to \eqref{equation:MontagnaForTheories} --Montagna's principle-- we can conclude that 
\[
(A\rhd B) \to \big( (A \wedge \Box C) \rhd (B\wedge \Box C) \big) \in \intl{T}
\]
whenever $T$ proves full induction. However, there is no direct reflection of Pudl\'ak's principle on the level of interpretability logics since Pudl\'ak's principle would translate to 
\[
(A \rhd B ) \to \big( (A \wedge \Box^J C) \rhd (B\wedge \Box C) \big) \in \intl{T}
\]
for the particular cut $J$ corresponding to $j:A\rhd B$ and this cannot be expressed in our modal language. In a sense, $\Box^J C$ corresponds to finding a small witness of the provability of $C$. As we shall see, there are various occasions where we can conclude that such small witnesses exist. The two main ingredients in obtaining such small witnesses are expressed by the following lemmas.

\begin{lemma}[Outside big, inside small]
For $T, U$ any theories extending \sonetwo, we have that 
\[
T\vdash \forallcut J\,  \forall x\  \Box_U (\dot x \in \dot J ).
\]
\end{lemma}

\begin{proof}
Given $J$ and given $x$, not necessarily in $J$, we can construct a proof-object to the extent that $x\in J$ in the obvious way. First conclude $J(0)$ which holds since $J$ is a cut. Next, conclude that $J(1)$ from the progressiveness of $J$ and $J(0)$ and so all the way to $J(x)$. This proof object is not much bigger than $x$ itself. However it requires the totality of exponentiation. If this is not provable in $T$, the proof can be generalized by switching do dyadic numerals and we refer to e.g.~\cite{Buss:1986:BoundedArithmetic, JoostenVisser:2000:IntLogicAll} for details.
\end{proof}

\begin{lemma}[Formalized Henkin construction]
For theories $T, U$ and $V$ all extending \sonetwo we have
\[
T\vdash \forallcut J\   (U \cup \{ {\sf Con}^J (V) \} \rhd V).
\]
\end{lemma}

\begin{proof}
(Sketch) The theory $T$ can verify that the usual Henkin construction can be formalized in $U$ without many problems where $J$ plays the role of the natural numbers. Instead of applying induction to obtain a maximal consistent set $\mathcal M_V$ as a consistent branch of infinite length in Lindenbaum's lemma, we can now only conclude that the length of the branch is within some cut $I$ which is a shortening of $J$ thereby yielding a set $\mathcal M_V^I$ which is contradiction-free on $I$. 

The set $\mathcal M_V^I$ can be used to obtain a term model and we define an interpretation $j:(U \cup \{ {\sf Con}^J (V) \} \rhd V)$ from the term model as usual so that provably $\phi^j \leftrightarrow \big( \phi \in \mathcal M_V^I\big)$. Note that since the interpretation of identity can be any equivalence relation, there is no need to move to equivalence classes in the construction of our term model.  
By construction we have $\Box_U \forall \phi\ \big( {\sf Con}^J (V) \wedge \Box_V^I \phi \to \phi^j \big)$. By the outside big, inside small principle and the formalized deduction theorem we now conclude that
\[
\forall \phi \ (\Box_V \varphi \to \Box_{U \cup \{ {\sf Con}^J (V) \} } \varphi)
\]
which, by \eqref{equation:definitionOfInterpretabilityFormalized} is nothing but $(U \cup \{ {\sf Con}^J (V) \} \rhd V)$. We refer to \cite{Visser:1991:FormalizationOfInterpretability} where one can see that the necessary induction for this argument is available in \sonetwo.
\end{proof}

Using these lemmas we can infer in various occasions the existence of small witnesses to provability.

\begin{lemma}\label{theorem:negatedRhdYieldsSmallWitness}
For any theory $T$ we have $T\vdash \neg (A\rhd \neg C) \to \forallcut K \Diamond (A\wedge \Box^{\dot K} C)$.
\end{lemma}

\begin{proof}
Reason in arbitrary $T$ by contraposition and apply the Henkin construction on a cut.
\end{proof}

As a corollary to this lemma, we see that $(A\rhd B) \to \big( \neg (A\rhd \neg C) \rhd (B\wedge \Box C) \big) \in \intl{T}$ for any $T$ extending \sonetwo. It is an open problem to classify the modal principles that hold in any theory extending \sonetwo. This paper raises the previously known lower bound.
 
We formulate some other direct corollaries of the outside-big inside-small principle in the following useful lemma.

\begin{lemma}\label{theorem:K4WithCuts}
Let $T$ be any theory containing \sonetwo. We have that 
\begin{enumerate}
\item
$T\vdash \forall A\ (\, \Box \dot A \to \forallcut K\  \Box \Box^{\dot K} \dot A \, )$;

\item
$T\vdash \sigma \to \forallcut K\  \Box \sigma^{\dot K}$ for any formula $\sigma$ in $\exists \Sigma_1^b$;

\item
$T\vdash \forall \, C {\in} \exists \Sigma_1^b \, \forallcut J \ ( \exists x \, C \to \Box \, \exists\, x{\in}{\dot J} \, C)$.

\end{enumerate}
\end{lemma}

One ingredient in proving interpretability principles arithmetically sound, is to find small witnesses. Another ingredient tells us how we can \emph{keep} these witnesses small. A simple generalization of Pudl\'ak's lemma which was first proved in \cite{Joosten:2015:OreyHajek} and tells us how to do so.

\begin{lemma}\label{theorem:KeepWitnessSmall}
If $j: \alpha \rhd \beta$ then, for every cut $I$ there exists a definable cut $J$ such that for every $\gamma$ we have that 
\[
T \vdash  \forallcut I\, \existscut J\, \exists^{\sf int} j\ \Big(j: (\alpha \wedge \Box^J \gamma) \rhd (\beta \wedge \Box^I \gamma )\Big).
\] 
\end{lemma}

\subsection{Modal interpretability logics}

When working in interpretability logic, we shall adopt a reading convention that will allow us to omit many brackets. Thus, we say that the strongest binding `connectives' are $\neg$, $\Box$ and $\Diamond$ which all bind equally strong. Next come $\wedge$ and $\vee$, followed by $\rhd$ and the weakest connective is $\to$. Thus, for example, $A\rhd B \to A \wedge \Box C \rhd B\wedge \Box C$ will be short for $(A\rhd B) \to \big((A \wedge \Box C) \rhd (B\wedge \Box C)\big)$. 

If we do not disambiguate a formula of nested conditionals ($\to$ or $\rhd$), then this should be read as a conjunction. For example, $A\rhd B\rhd C$ should be read as $(A\rhd B) \wedge (B\rhd C)$ and likewise for implications. 

We first define the core logic \il which shall be present in any other interpretability logic. As before, we work in a propositional signature where apart from the classical connectives we have a unary modal operator $\Box$ and a binary modal operator $\rhd$.

\begin{definition}[\il]
The logic \il contains apart from all propositional logical tautologies, all instantiations of the following axiom schemes.

\begin{enumerate}
\item[${\sf L1}$]\label{ilax:l1}
        $\Box(A\rightarrow B)\rightarrow(\Box A\rightarrow\Box B)$
\item[${\sf L2}$]\label{ilax:l2}
        $\Box A\rightarrow \Box\Box A$
\item[${\sf L3}$]\label{ilax:l3}
        $\Box(\Box A\rightarrow A)\rightarrow\Box A$
\item[${\sf J1}$]\label{ilax:j1}
        $\Box(A\rightarrow B)\rightarrow A\rhd B$
\item[${\sf J2}$]\label{ilax:j2}
        $(A\rhd B)\wedge (B\rhd C)\rightarrow A\rhd C$
\item[${\sf J3}$]\label{ilax:j3}
        $(A\rhd C)\wedge (B\rhd C)\rightarrow A\vee B\rhd C$
\item[${\sf J4}$]\label{ilax:j4}
        $A\rhd B\rightarrow(\Diamond A\rightarrow \Diamond B)$
\item[${\sf J5}$]\label{ilax:j5}
        $\Diamond A\rhd A$
\end{enumerate}

The rules of the logic are Modus Ponens (from $A\to B$ and $A$, conclude $B$) and Necessitation (from $A$ conclude $\Box A$).
\end{definition}

It is not hard to see that $\il \subseteq \ilal$. By \ilm we denote the logic that arises by adding Montagna's axiom scheme 
\[
\principle M \ :\ \ \ \ A \rhd B \rightarrow A \wedge \Box C \rhd B \wedge \Box C
\]
to \il. It follows from our earlier observations that $\ilm \subseteq \intl T$ and the other inclusion can be proven too. 
\begin{theorem}[Berarducci \cite{Berarducci:1990:InterpretabilityLogicPA}, Shavrukov \cite{Shavrukov:1988:InterpretabilityLogicPA}]\label{theo:shav}
If $T$ proves full induction, then $\intl{T}=\ilm$.
\end{theorem}
The logic \ilp arises by adding the axiom scheme
\[
\principle P \ :\ \ \ \ A \rhd B \rightarrow \Box (A \rhd B)
\]
to the basic logic \il. If $T$ is finitely axiomatizable it is easy to see that \eqref{equation:definitionOfInterpretabilityFormalized} is provably equivalent to a $\Sigma_1$ formula so that by provable $\Sigma_1$ completeness we see that $\ilp \subseteq \intl T$ for any finitely axiomatized theory $T$ that proves that exponentiation is a total function. If $T$ can moreover prove the totality of superexponentiation ${\tt supexp}$ then the inclusion can be reversed too. Here, ${\tt supexp}(x)$ is defined as $x\mapsto 2^x_x$ with $2^n_0:=n$ and $2^{n}_{m+1}:= 2^{(2^n_m)}$.

\begin{theorem}[Visser \cite{Visser:1990:InterpretabilityLogic}]
If $T$ is finitely axiomatizable and proves the totality of ${\tt supexp}$, then $\intl{T}=\ilp$.
\end{theorem}

It follows that $\il \subseteq \ilal \subseteq (\ilp \cap \ilm)$. In this paper we shall focus on these bounds.

\subsection{Relational semantics}

We can equip interpretability logics with a natural relational semantics often referred to as Veltman semantics.

\begin{definition}
A \emph{Veltman frame} is a triple $\la W, R, \{S_x\}_{x\in W}\ra$ where $W$ is a non-empty set of \emph{possible worlds}, $R$ a binary relation on $W$ so that $R^{-1}$ is transitive and well-founded. The $\{S_x\}_{x\in W}$ is a collection of binary relations on $x\uparrow$ (where $x\uparrow:= \{ y\mid xRy \}$). The requirements are that the $S_x$ are reflexive and transitive and the restriction of $R$ to $x\uparrow$ is contained in $S_x$, that is $R\cap (x\uparrow) \subseteq S_x$. 

A \emph{Veltman model} consists of a Veltman frame together with a valuation $V : {\tt Prop} \to \mathcal P (W)$ that assigns to each propositional variable $p\in {\tt Prop}$ a set of worlds $V(p)$ in $W$ where $p$ is stipulated to be true. This valuation defines a forcing relation $\Vdash \ \subseteq W{\times} {\sf Form}$ telling us which formulas are true at which particular world:
\[
\begin{array}{lll}
x\Vdash \bot & & \mbox{for no $x\in W$};\\
x\Vdash A\to B & :\Leftrightarrow & x\nVdash  A \mbox{ or } x\Vdash B;\\
x\Vdash \Box A & :\Leftrightarrow & \forall y\ ( xRy \to y\Vdash  A);\\
x\Vdash A\rhd B & :\Leftrightarrow & \forall y\ \Big( xRy \wedge y\Vdash A \to \exists z \ (yS_xz \wedge z\Vdash  B\Big).
\end{array}
\]
For a Veltman model $\mathcal M = \la W, R, \{S_x\}_{x\in W}, V\ra$, we shall write $\mathcal M \models A$ as short for $\forall \, x {\in} W\ \mathcal M, x \Vdash A$.
\end{definition}
The logic \il is sound and complete with respect to all Veltman models (\cite{JonghVeltman:1990:ProvabilityLogicsForRelativeInterpretability}). Often one is interested in considering all models that can be defined over a frame. Thus, given a frame $\mathcal F$ and a valuation $V$ on $\mathcal F$ we shall denote the corresponding model by $\la \mathcal F, V \ra$. A \emph{frame condition} for a modal formula $\sf P$ is a formula $F$ (first or higher-order) in the language $\{ R, \{S_x\}_{x\in W} \}$ so that $\mathcal F \models F$ (as a relational structure) if and only if $\forall^{\sf valuation} V\ \la \mathcal F,V\ra \models {\sf P}$.

It is easy to establish that the frame condition for $\principle P$ is $xRyRzS_xu \to zS_yu$ where $xRyRzS_xu$ is short for $xRy \wedge yRz \wedge zS_xu$. Likewise, it is elementary to see that the frame condition for \principle M is given by $yS_xzRu\to yRu$. In this paper we shall compute the frame conditions for two new series of principles in \ilal.

Often we shall denote a valuation $V$ directly by the induced forcing relation $\Vdash$. Given a Veltman model $\la \mathcal F, \Vdash \ra$ we define a \emph{$C$-assuring successor} --denoted by $R^C_\Vdash$-- as follows
\[
xR^C_\Vdash y \ \ := \ \ \big(  xRy \wedge  y \Vdash C \ \wedge \ \forall z\ (yS_x z \to z\Vdash C)  \big).
\]

\section{A slim hierarchy of principles}

In this section we present a hierarchy of interpretability principles in \ilal of growing strength. For a well-behaved sub-hierarchy we shall compute the frame conditions and prove arithmetical soundness. There is no particular `slimness' inherent to the hierarchy presented here. The main reason for our name is that we tend to depict the frame conditions (see Figure \ref{figure:SlimHierarchy}) in a slim way as opposed to the depicted frame conditions for the series of principles that we refer to as a broad series of principles (see Figure \ref{figure:broadSeries}).

\subsection{A slim hierarchy}

Inductively, we define a series of principles as follows.

\[
\begin{array}{lll}
\principle{R_0} &:=& A_0 \rhd B_0 \to \neg (A_0 \rhd \neg C_0) \rhd B_0 \wedge \Box C_0\\
\ & \ & \ \\ 
\principle{R_{2n+1}}&:= & R_{2n} [\neg (A_n\rhd \neg C_n) /\neg (A_n\rhd \neg C_n) \wedge (E_{n+1}\rhd \Diamond A_{n+1});\\
\ & \ & 
B_n \wedge \Box C_n/B_n \wedge \Box C_n \wedge (E_{n+1} \rhd A_{n+1})]\\
\ & \ & \ \\ 
\principle{R_{2n+2}}&:=& R_{2n+1} [B_n/ B_n \wedge (A_{n+1}\rhd B_{n+1});\\
\ & \ & 
\Diamond A_{n+1} / \neg (A_{n+1}\rhd \neg C_{n+1});\\
\ & \ & 
(E_{n+1}\rhd A_{n+1})/ (E_{n+1}\rhd A_{n+1}) \wedge (E_{n+1} \rhd B_{n+1}\wedge \Box C_{n+1})
]\\
\end{array}
\]
As to illustrate how these substitutions work we shall calculate the first five principles.
\[
\begin{array}{lll}
\principle{R_0} &:=& A_0 \rhd B_0 \to \neg (A_0 \rhd \neg C_0) \rhd B_0 \wedge \Box C_0\\

\principle{R_1} & := & A_0 \rhd B_0 \to \neg (A_0 \rhd \neg C_0) \wedge (E_1 \rhd \Diamond A_1) \rhd B_0 \wedge \Box  C_0  \wedge (E_1\rhd A_1)\\

\principle{R_2} &:=& A_0 \rhd B_0 \wedge (A_1 \rhd B_1) \to \neg (A_0 \rhd \neg C_0) \wedge (E_1 \rhd \neg (A_1\rhd \neg C_1)) \ \rhd\\
\ & &\ \ \ \ \ \ \ \ \ \ \ \  \ \ B_0 \wedge (A_1 \rhd B_1) \wedge \Box  C_0  \wedge (E_1\rhd A_1) \wedge (E_1\rhd B_1 \wedge \Box C_1)\\

\principle{R_3} &:=& A_0 \rhd B_0 \wedge (A_1 \rhd B_1) \to \\
 & & \ \neg (A_0 \rhd \neg C_0) \wedge (E_1 \rhd \neg (A_1\rhd \neg C_1)\wedge (E_2 \rhd \Diamond A_2)) \ \rhd\\
\ & &\ \ \  B_0 \wedge (A_1 \rhd B_1) \wedge \Box  C_0  \wedge (E_1\rhd A_1) \wedge (E_1\rhd B_1 \wedge \Box C_1 \wedge (E_2\rhd A_2))\\

\principle{R_4} &:=& A_0 \rhd B_0 \wedge (A_1 \rhd B_1\wedge (A_2\rhd B_2)) \to \\
 & & \ \neg (A_0 \rhd \neg C_0) \wedge (E_1 \rhd \neg (A_1\rhd \neg C_1)\wedge (E_2 \rhd \neg( A_2\rhd \neg C_2))) \ \rhd\\
\ & &\ \ \  B_0 \wedge (A_1 \rhd B_1\wedge (A_2\rhd B_2)) \wedge \Box  C_0  \wedge (E_1\rhd A_1) \ \wedge \\
\ & &\ \ \  \big(E_1\rhd B_1\wedge (A_2\rhd B_2) \wedge \Box C_1 \wedge (E_2\rhd A_2)\wedge (E_2\rhd B_2 \wedge \Box C_2)\big)\\
\end{array}
\]

It is easy to see that the hierarchy defines a series of principles of increasing strength as expressed by the following lemma.

\begin{lemma}
For each natural number $n$ we have that $\extil{R_{n+1}}\vdash {\sf R_{n}}$.
\end{lemma}

\begin{proof}
By an easy case distinction. We see that $\vdash_{\sf IL} {\principle{R}_{2n+1}} \to {\principle R_{2n}}$ by choosing $E_{n+1} := \Diamond \top$ and $A_{n+1}:= \top$. To see that 
$\vdash_{\sf IL} {\principle R_{2n+2}} \to {\principle R_{2n+1}}$ we choose $C_{n+1} := \top$ and $B_{n+1}:=A_{n+1}$.
\end{proof}

Thus, to understand the hierarchy well, it suffices to study a well-behaved co-final subsequence of it. To this end we define the following hierarchy.

For any $n\geq 0$ we define schemata $\xx_n$, $\yy_n$ and $\zz_n$ as follows.
\begin{align*}
%\xx_0 &= A_0\rhd B_0\enspace,\\
%\yy_0 &= \neg(A_0\rhd C_0)\enspace,\\
%\zz_0 &= B_0\wedge\Box C_0\enspace.\\
\xx_0 &= A_0\rhd B_0;\\
\yy_0 &= \neg(A_0\rhd C_0);\\
\zz_0 &= B_0\wedge\Box C_0;\\
%\end{align*}
&\\
%\begin{align*}
\xx_{n+1} &= A_{n+1}\rhd B_{n+1}\wedge (\xx_n);\\
\yy_{n+1} &= \neg(A_{n+1}\rhd \neg C_{n+1})\wedge(E_{n+1}\rhd \yy_n);\\
\zz_{n+1} &= B_{n+1}\wedge (\xx_n)\wedge\Box C_{n+1}\wedge(E_{n+1}\rhd A_n)\wedge(E_{n+1}\rhd \zz_n).
\end{align*}
For any $n\geq 0$ define
\[
\rr_n = \xx_n\rightarrow \yy_n\rhd \zz_n.
\]
To see how this proceeds, let us evaluate the first couple of instances:
\[
\begin{array}{lll}
\rr_0 & := & A_0 \rhd B_0 \to \neg (A_0 \rhd \neg C_0)\rhd B_0 \wedge \Box C_0;\\
\rr_1 & := & A_1 \rhd B_1\wedge (A_0\rhd B_0) \to \\
 & &\ \ \ \ \neg (A_1 \rhd \neg C_1)\wedge (E_1 \rhd \neg (A_0 \rhd \neg C_0))\ \rhd \\
 & & \ \ \ \ \ \ \ \ \ \ \ \ \ \ \ \ \  B_1 \wedge (A_0\rhd B_0) \wedge \Box C_1 \wedge (E_1 \rhd A_0)\wedge (E_1\rhd B_0 \wedge \Box C_0);\\
\rr_2 & := & A_2 \rhd B_2 \wedge (A_1 \rhd B_1\wedge (A_0\rhd B_0)) \to \\
 & &\ \ \ \ \neg (A_2 \rhd \neg C_2) \wedge \big(E_2\rhd \neg (A_1 \rhd \neg C_1)\wedge (E_1 \rhd \neg (A_0 \rhd \neg C_0))\big)\ \rhd \\
 & & \ \ \ \ \ \ \ \ \ \ \ \ \ \ B_2 \wedge (A_1 \rhd B_1\wedge (A_0\rhd B_0)) \wedge \Box C_2 \wedge (E_2\rhd A_1)\ \wedge \\
 & & \ \ \ \ \ \ \ \ \ \ \ \ \ \   \big(E_2\rhd  B_1 \wedge (A_0\rhd B_0) \wedge \Box C_1 \wedge (E_1 \rhd A_0)\wedge (E_1\rhd B_0 \wedge \Box C_0)\big);\\

\end{array}
\]
It is clear that the $\rr_k$ hierarchy is directly related to the $\principle R_k$ hierarchy:
\begin{lemma}\label{thm:subhierarchy}
For each natural number $k$ we have  $\principle R_{2k} := \rr_k [\mathbb X_i/ \mathbb X_{k-i};\, E_i/E_{k+1-i}]$, where $\mathbb X \in \{  A, B, C\}$.
\end{lemma}

\begin{proof}
By visual inspection we see that it holds for $k=0,1$. It is proven in full generality by an easy induction. To prove the lemma, it is best to consider the place-holders like $A_i$ etc.~as propositional variables since otherwise in principle, for example, $A_i$ could contain $E_i$ as a subformula.
\end{proof}

For the remainder of this section, we shall focus on the $\rr_k$ hierarchy and begin by computing a collection of frame condititions.

\subsection{Frame conditions}

For any $n\geq0$ we define a ternary relation $\mathcal G_n(x,y,z)$ on Veltman-frames as follows.
\begin{align*}
\mathcal G_0(x,y,z) &= \forall u\, (zRu\Rightarrow yS_xu),\\
\mathcal G_{n+1}(x,y,z) &= \forall u\, \big(zRu\Rightarrow yS_xu\wedge \forall v\, (uS_xv\Rightarrow \mathcal G_n(z,u,v)\big).
\end{align*}
For every $n\geq0$ we define the first-order frame condition $\mathcal F_n$ as follows.
\[
\mathcal F_n = \forall\, w,x,y,z\ (wRxRyS_wz\Rightarrow \mathcal G_n(x,y,z)).
\]
The main result of this subsection is that $\mathcal F_{2n}$ is the frame correspondence of $\rr_n$. For $n=0$ this has been established in \cite{GorisJoosten:2011:ANewPrinciple}. It is easy to see that $\mathcal G_{n+1}(x,y,z)$ implies $\mathcal G_{n}(x,y,z)$ so that $\mathcal F_{n+1}$ also implies $\mathcal F_n$. The frame conditions $\mathcal F_k$ are depicted in Figure \ref{figure:SlimHierarchy} for the first three values of $k$.

\begin{figure}[h]
\begin{picture}(95,90)(0,0)
\unitlength=1,4mm
\put(-40,00){
\gasset{ExtNL=y,NLdist=0.2,Nw=1,Nh=1}
\node[NLangle=180](W)(50,0){$w$}
\node[NLangle=180](X0)(40,10){$x_0$}
\node[NLangle=180](Y0)(50,20){$y_0$}
\node[NLangle=0](X1)(60,20){$x_1$}
\node[NLangle=0](Y1)(50,30){$y_1$}
\drawedge(W,X0){}
\drawedge(X0,Y0){}
\drawedge(W,X1){}
\drawedge(X1,Y1){}
\drawbpedge[ELpos=30,ELdist=0.2](Y0,45,5,X1,225,5){$S_w$}
\drawbpedge[ELpos=60,ELdist=0.2,dash={1.5}0](Y0,45,5,Y1,225,5){$S_{x_0}$}
}
\put(-5,0){
%\unitlength=2mm
\gasset{ExtNL=y,NLdist=0.2,Nw=1,Nh=1}
\node[NLangle=180](W)(50,0){$w$}
\node[NLangle=180](X0)(40,10){$x_0$}
\node[NLangle=180](Y0)(50,20){$y_0$}
\node[NLangle=0](X1)(60,20){$x_1$}
\node[NLangle=0](Y1)(50,30){$y_1$}
\node[NLangle=180](X2)(40,30){$x_2$}
\node[NLangle=180](Y2)(50,40){$y_2$}
\drawedge(W,X0){}
\drawedge(X0,Y0){}
\drawedge(W,X1){}
\drawedge(X1,Y1){}
\drawedge(X0,X2){}
\drawedge(X2,Y2){}
\drawbpedge[ELpos=30,ELdist=0.2](Y0,45,5,X1,225,5){$S_w$}
\drawbpedge[ELpos=60,ELdist=0.2](Y0,45,5,Y1,225,5){$S_{x_0}$}
\drawbpedge[ELside=r,ELpos=30,ELdist=0.2](Y1,135,5,X2,315,5){$S_{x_0}$}
\drawbpedge[ELside=r,ELpos=60,ELdist=0.2,dash={1.5}0](Y1,135,5,Y2,315,5){$S_{x_1}$}
}
\put(30,0){
%\unitlength=2mm
\gasset{ExtNL=y,NLdist=0.2,Nw=1,Nh=1}
\node[NLangle=180](W)(50,0){$w$}
\node[NLangle=180](X0)(40,10){$x_0$}
\node[NLangle=180](Y0)(50,20){$y_0$}
\node[NLangle=0](X1)(60,20){$x_1$}
\node[NLangle=0](Y1)(50,30){$y_1$}
\node[NLangle=180](X2)(40,30){$x_2$}
\node[NLangle=180](Y2)(50,40){$y_2$}
\node[NLangle=180](X3)(60,40){$x_3$}
\node[NLangle=180](Y3)(50,50){$y_3$}
\drawedge(W,X0){}
\drawedge(X0,Y0){}
\drawedge(W,X1){}
\drawedge(X1,Y1){}
\drawedge(X0,X2){}
\drawedge(X2,Y2){}
\drawedge(X1,X3){}
\drawedge(X3,Y3){}
\drawbpedge[ELpos=30,ELdist=0.2](Y0,45,5,X1,225,5){$S_w$}
\drawbpedge[ELpos=60,ELdist=0.2](Y0,45,5,Y1,225,5){$S_{x_0}$}
\drawbpedge[ELside=r,ELpos=30,ELdist=0.2](Y1,135,5,X2,315,5){$S_{x_0}$}
\drawbpedge[ELside=r,ELpos=60,ELdist=0.2](Y1,135,5,Y2,315,5){$S_{x_1}$}
\drawbpedge[ELpos=30,ELdist=0.2](Y2,45,5,X3,225,5){$S_{x_1}$}
\drawbpedge[ELside=l,ELpos=60,ELdist=0.2,dash={1.5}0](Y2,135,5,Y3,315,5){$S_{x_2}$}
}
\end{picture}
\caption{From left to right we have depicted $\mathcal F_0$ to $\mathcal F_2$. Since $\mathcal F_{k+1}$ implies $\mathcal F_k$ we have only depicted the content of $\mathcal F_{k+1}$ which is new w.r.t.~$\mathcal F_k$. As such we should read the pictures as: ``if all un-dashed relations are as in the picture, then also the dashed relation should be present".}\label{figure:SlimHierarchy}
\end{figure}
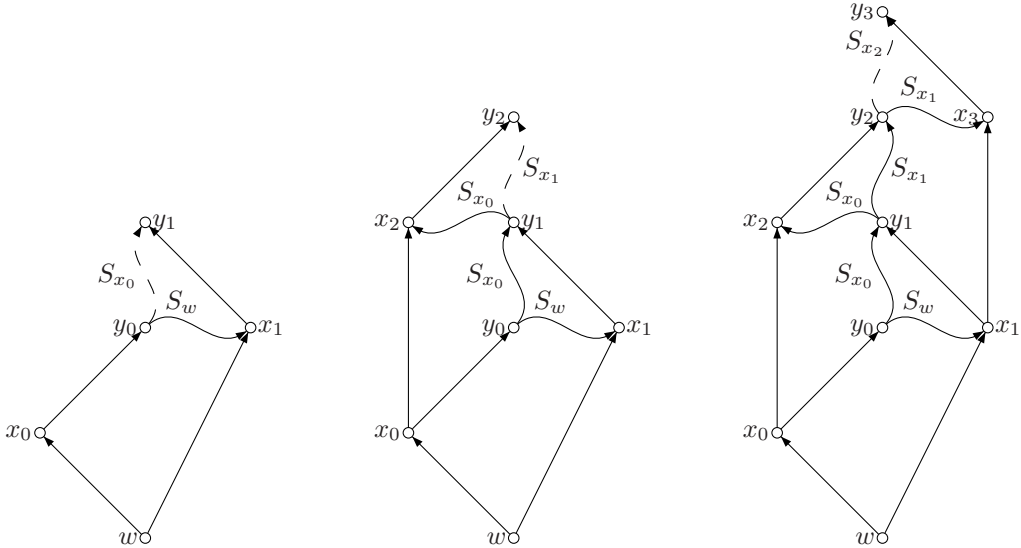

In what follows we let $F=\langle W,R,S\rangle$ be an arbitrary Veltman-frame.
With a forcing relation $\Vdash$ we will always mean a forcing relation on $F$.
For our convenience we define
\[
A_{-1}\equiv \xx_{-1}\equiv \zz_{-1}\equiv \top.
\]
Before we can prove a frame correspondence we first need a technical lemma.

\begin{lemma}\label{lemm:formpreserve0}
For all $k\geq 0$ and all $x,y,z\in W$.
If $\mathcal G_{2k}(x,y,z)$ then for any forcing relation $\Vdash$ for which
\[
x\Vdash \yy_k\quad\textrm{and}\quad xR^{C_k}_\Vdash y\quad\textrm{and}\quad z\Vdash \xx_{k-1},
\]
we also have
\[
z\Vdash\Box C_k\wedge(E_k\rhd A_{k-1})\wedge(E_k\rhd \zz_{k-1}).
\]
\end{lemma}

\begin{proof}
We shall write $xR^{C_k}y$ as short for $xR^{C_k}_\Vdash y$ and prove the claim by induction on $k$. With the convention that $A_{-1}\equiv \zz_{-1}\equiv\top$ the lemma is trivial for $k=0$.
So assume $k>0$. Let $\Vdash$ be a forcing relation and take $x$, $y$ and $z$ such that
\begin{align}
\label{l1} &x\Vdash\neg(A_k\rhd \neg C_k)\wedge(E_k\rhd \yy_{k-1}),\\
\label{l2} &xR^{C_k}y,\\
\label{l3} &z\Vdash \xx_{k-1},\\
\label{l4} &\mathcal G_{2k}(x,y,z).
\end{align}
Take an arbitrary $u\in W$ with $zRu$. By \eqref{l4} we have $yS_xu$ and thus by \eqref{l2} we have $u\Vdash C_k$.
This shows $z\Vdash\Box C_k$.

To show that also the other two conjuncts hold at $z$ assume that $u\Vdash E_k$.
By \eqref{l1} we find some $v$ with $uS_xv$ and
\begin{equation}\label{eq:ih0}
v \Vdash \yy_{k-1}.
\end{equation}
In order to show $z\Vdash E_k\rhd A_{k-1}$ we have to find some $a$ with $uS_za\Vdash A_{k-1}$.
Remark that $\yy_{k-1}$ implies $\Diamond A_{k-1}$ thus there exists some $a$ with $vRa\Vdash A_{k-1}$.
By \eqref{l4} we have $\mathcal G_{2k-1}(z,u,v)$ and thus $uS_za$.

In order to show that also $z\Vdash E_k\rhd \zz_{k-1}$ we have to find some $b$ with $uS_zb\Vdash \zz_{k-1}$.
We just used that $\yy_{k-1}$ implies $\Diamond A_{k-1}$, but remark that $\yy_{k-1}$ implies the stronger statement that $\neg(A_{k-1}\rhd \neg C_{k-1})$. Thus there exists some $a$ with
$a\Vdash A_{k-1}$ and 
\begin{equation}\label{eq:ih1}
vR^{C_{k-1}}a.
\end{equation}
As above, by \eqref{l4} we have $\mathcal G_{2k-1}(z,u,v)$ and thus
$uS_za$ and $zRa$.
By \eqref{l3} there exists a $b$ with $aS_zb$ and
\begin{equation}\label{eq:ih2}
b\Vdash B_{k-1}\wedge (\xx_{k-2}).
\end{equation}
Since $uS_za$ whence also $uS_zb$ holds, we will be done if we show that $b\Vdash \zz_{k-1}$.
%First we show that $b\Vdash\Box C_{k-1}$.
%So assume that there exists some $c$ with $bRc\Vdash C_{k-1}$ then since $G_{2k-2}(v,a,b)$ holds
%we also have $aS_vc$. But this would be in contradiction with $vR^{C_{k-1}}a$. Thus $z\Vdash\Box\neg C_{k-1}$.
To show that the remaining conjuncts of $\zz_{k-1}$ hold at $b$
(that is $b\Vdash \Box C_{k-1}\wedge (E_{k-1}\rhd A_{k-2})\wedge(E_{k-1}\rhd \zz_{k-2})$) simply observe that $\mathcal G_{2k-2}(v,a,b)$ and use \eqref{eq:ih0}, \eqref{eq:ih1} and
\eqref{eq:ih2} to invoke the (IH) on $v$, $a$ and $b$.
\end{proof}

\begin{corollary}\label{corr:frameToForm}
If $F\models \mathcal F_{2k}$, then $F\models\rr_k$.
\end{corollary}
\begin{proof}
Fix a forcing relation $\Vdash$ and let $w,x\in W$
such that $w\Vdash \xx_k$ and $wRx\Vdash \yy_k$.
Then for some $y$ we have $xR^{C_k}y\Vdash A_k$.
Thus there exists $z$ with $yS_wz$ and
\begin{equation}\label{eq:eq5}
z\Vdash B_k\wedge (\xx_{k-1})
\end{equation}
(recall $\xx_{-1}\equiv\top$).
Since $F\models \mathcal F_{2k}$ we have $\mathcal G_{2k}(x,y,z)$.
Thus by Lemma \ref{lemm:formpreserve0} we get
\begin{equation}\label{eq:eq6}
z\Vdash\Box C_k\wedge (E_k\rhd A_{k-1})\wedge(E_k\rhd \zz_{k-1})\enspace.
\end{equation}
Combining \eqref{eq:eq5} and \eqref{eq:eq6} gives $z\Vdash \zz_k$.
\end{proof}

The reversal of this corollary is again preceded by a technical lemma. We shall denote by $\bm a_k$, ${\bm b}_k$, ${\bm c}_k$, and ${\bm e}_k$, propositional variables that shall play the role of the $A_k$, $B_k$, $C_k$ and $E_k$ respectively in the principles $\rr_n$. Likewise, by $\xxp_k$ we shall denote the formula that arises by substituting $\bm a_j$ for $A_j$ in $\xx_k$ and $\bm b_j$ for $B_j$. The formulas $\yyp_k$ and $\zzp_k$ are defined similarly.

\begin{lemma}\label{lemm:lemm1}
For any $k\geq 0$ and all $x,y,z\in W$. If for all forcing relations $\Vdash$ for which
\[
x\Vdash \yyp_k\textrm{ and }xR^{\bm c_k}_\Vdash y\textrm{ and }z\Vdash \xxp_{k-1}
\]
we also have
\[
z\Vdash \Box \bm c_k\wedge(\bm e_k\rhd \bm a_{k-1})\wedge (\bm e_k\rhd \zzp_{k-1}),
\]
then $\mathcal G_{2k}(x,y,z)$.
\end{lemma}

\begin{proof}
Induction on $k$. Let $x,y,z\in W$ and assume the conditions of the lemma.
Unfolding the definition of $\mathcal G_{2k}(x,y,z)$ shows us that we have to show that
\begin{enumerate}
\item\label{p1} for all $u$ with $zRu$ we have $yS_xu$ \ \ \ \ \ \ \ \ \  \ \ \ \ \ \ \ \ \  \ \ \ \ \ \ \ \ \  \ \ \ \ \ \ \ \ \  \ \  ($k\geq 0$);
\item\label{p2} and for all $v$ and $a$ with $uS_xv$ and $vRa$ we have $uS_za$   \ \ \ \ \ \ \ \ \ \ \ ($k> 0$);
\item\label{p3} and for all $b$ with $aS_zb$ we have $\mathcal G_{2(k-1)}(v,a,b)$  \ \  \ \ \ \ \ \ \ \ \  \ \ \ \ \ \ \ \ \ ($k> 0$).
\end{enumerate}
We will show \ref{p1} and \ref{p2} `by hand' and invoke the (IH) for \ref{p3}.
In each of the three cases we will choose similar but different forcing relations $\Vdash$.

We first show \ref{p1}.
So let $zRu$. Define
\[
w\Vdash  \bm c_k\Leftrightarrow yS_xw \quad \textrm{ and } \quad w\Vdash \bm a_k\Leftrightarrow w=y.
\]
And let all the other variables be false everywhere.
Then $xR^{\bm c_k}_\Vdash y$ and $x\Vdash\neg(\bm a_k\rhd \neg \bm c_k)$. 
Since none of the $\bm e_i$ nor $\bm a_j$ with $j\neq k$ holds anywhere in the model, we trivially have $x\Vdash \yyp_k$ and $z\Vdash \xxp_{k-1}$ and thus according to
the conditions of the lemma in particular $z\Vdash\Box \bm c_k$.
By definition of $\Vdash$ we thus have $yS_xu$ which proves \ref{p1}. Note that for $k=0$ we only have to look after \ref{p1} hence we have now dealt with the base case of our induction.

Now we continue to show \ref{p2} assuming $k>0$. Choose any $v$ and $a$ with $uS_xv$ and $vRa$.
As above define
\[
w\Vdash \bm c_k\Leftrightarrow yS_xw \quad \textrm{ and } \quad w \Vdash \bm a_k\Leftrightarrow w=y.
\]
We now also define
\[
\begin{array}{lr}
w\Vdash \bm e_k\Leftrightarrow w=u & \quad \textrm{ and, }\\ 
w\Vdash \bm a_{k-1} \Leftrightarrow w=a \Leftrightarrow w\Vdash \bm b_{k-1} & \quad \textrm{ and, } \\
w\Vdash \bm c_{k-1} \Leftrightarrow aS_z w.
\end{array}
\]

Let all the other propositional variables be false everywhere. Now $v\Vdash \yyp_{k-1}$ and thus $x\Vdash \yyp_k$. It is not hard to see that we also have $z\Vdash \xxp_{k-1}$ and thus
according to the condition of the lemma we have in particular $z\Vdash \bm e_k\rhd \bm a_{k-1}$.
Since $zRu\Vdash \bm e_k$ there must be an $a'$ with $uS_za'\Vdash \bm a_{k-1}$.
Since $a$ is the only world that forces $\bm a_{k-1}$ we must have $uS_za$.

To finish and show \ref{p3} choose $b$ such that $aS_zb$. We want to show that $\mathcal G_{2(k-1)}(v,a,b)$.
Invoking the (IH) it is enough to show that for any forcing relation $\Vdash$ for which
\begin{equation}\label{eq:c0}
v\Vdash \yyp_{k-1}, \quad \textrm{and}\quad vR^{\bm c_{k-1}}_\Vdash a\quad\textrm{and}\quad b\Vdash \xxp_{k-2},
\end{equation}
we also have
\begin{equation}\label{eq:chalf}
b\Vdash (\bm e_{k-1}\rhd \bm a_{k-2}) \wedge (\bm e_{k-1}\rhd \zzp_{k-2}) \wedge \Box\bm c_{k-1}.
\end{equation}
Our strategy in proving this is as follows. We slightly tweak $\Vdash$ to obtain $\Vdash'$. This $\Vdash'$ is similar to $\Vdash$ in that \eqref{eq:c0} still holds and moreover 
\begin{equation}\label{eq:TweakedVdashSimilar}
b\Vdash A \Leftrightarrow b\Vdash' A \quad \mbox{for subformulas $A$ of $(\bm e_{k-1}\rhd \bm a_{k-2}) \wedge (\bm e_{k-1}\rhd \zzp_{k-2}) \wedge \Box\bm c_{k-1}$}. 
\end{equation}
However, it is (possibly) different in that we now know that $x\Vdash' \yyp_k$, and $xR^{\bm c_k}_{\Vdash'}y$ and, $z\Vdash' \xxp_{k-1}$ so that we may apply the main assumption of the lemma to $\Vdash'$ concluding $z\Vdash' \Box \bm c_k\wedge(\bm e_k\rhd \bm a_{k-1})\wedge (\bm e_k\rhd \zzp_{k-1})$. The latter will help us conclude \eqref{eq:chalf}.

Thus we consider an arbitrary forcing relation $\Vdash$ that satisfies \eqref{eq:c0}.
We modify $\Vdash$ to obtain $\Vdash'$ such that it satisfies

\[
\begin{array}{lll}
w\Vdash' \bm a_k & \Leftrightarrow & w=y;\\
w\Vdash' \bm e_k & \Leftrightarrow & w=u;\\
w\Vdash'  \bm c_k & \Leftrightarrow & yS_xw;\\
w\Vdash'  \bm a_{k-1} & \Leftrightarrow & w=a;\\
w\Vdash' \bm b_{k-1} & \Leftrightarrow &  w=b.\\
\end{array}
\]
Apart from these modifications, $\Vdash'$ will coincide with $\Vdash$. It is a straightforward check to see that we have \eqref{eq:c0} for $\Vdash'$ and that moreover \eqref{eq:TweakedVdashSimilar} holds. In addition, by the definition of $\Vdash'$ we now also have
\begin{equation}\label{eq:c1}
x\Vdash' \yyp_k\quad\textrm{and}\quad xR^{\bm c_k}_{\Vdash'}y \quad\textrm{and} \quad z\Vdash' \xxp_{k-1}.
\end{equation}
Thus, we see that $\Vdash'$ satisfies the antecedent of the condition of the lemma.
Consequently, we have
$z\Vdash' \bm e_k\rhd \zzp_{k-1}$.
Since $zRu\Vdash' \bm e_k$, there must exist some $b'$ with $uS_zb'\Vdash'\zzp_{k-1}$.
But now, since $\bm b_{k-1}$ is a conjunct of $\zzp_{k-1}$ and $b$ is the only world that $\Vdash'$-forces $\bm b_{k-1}$, we must have $b\Vdash' \zzp_{k-1}$. 
In particular, we conclude $b\Vdash' (\bm e_{k-1}\rhd \bm a_{k-2})\wedge(\bm e_{k-1}\rhd \zzp_{k-2})\wedge\Box \bm c_{k-1}$; by \eqref{eq:TweakedVdashSimilar} the same holds for $\Vdash$ and we are done.
\end{proof}

Putting this all together gives us the frame correspondence for $\rr_k$.

\begin{theorem}
%If $F\models \xx_k\rightarrow \yy_k\rhd \zz_k$ then $F\models F_{2k}$
For any Veltman frame $F$ and any natural number $k\geq 0$ we have 
\[
F\models \mathcal F_{2k} \ \Longleftrightarrow \ F\models\rr_k  \ \Longleftrightarrow \ F\models \principle{R}_{2k}.
\]
\end{theorem}

\begin{proof}
The second equivalence is a direct consequence of Lemma \ref{thm:subhierarchy} so we focus on the first equivalence. 

The $\Rightarrow$ direction is just Corollary \ref{corr:frameToForm}. For the other direction, fix some $k$, assume that $F\models\rr_k$ and let $wRxRyS_wz$. We have to show that $\mathcal G_{2k}(x,y,z)$.
Now consider any forcing relation $\Vdash$ that satisfies $xR^{\bm c_k}_\Vdash y$, and $x\Vdash \yyp_k$ and, $z\Vdash \xxp_{k-1}$. 
By Lemma \ref{lemm:lemm1} it is enough to show that
\begin{equation}\label{eq:eq7}
z\Vdash \Box \bm c_k \wedge (\bm e_k\rhd \bm a_{k-1}) \wedge (\bm e_k\rhd\zzp_{k-1}).
\end{equation}
Now consider a forcing relation $\Vdash'$ where $\Vdash'$ is like $\Vdash$ except that
\[
v\Vdash' \bm a_k \Leftrightarrow v=y \quad\textrm{and}\quad v\Vdash' \bm b_k\Leftrightarrow v=z.
\]
Notice that $xR^{\bm c_k}_{\Vdash'}y$ and thus also $x\Vdash' \yyp_k$.
But now we have $w\Vdash' \xxp_k$ as well and thus $w\Vdash' \yyp_k\rhd \zzp_k$.
Thus there must be some $z'$ with $xS_wz'\Vdash\zzp_k$.
Since $\bm b_k$ is a conjunct of $\zzp_k$ and $z$ is the only world where $\bm b_k$ is forced we must have
$z\Vdash' \zzp_k$. Since $\Box \bm c_k\wedge(\bm e_k\rhd \bm a_{k-1})\wedge(\bm e_k\rhd\zzp_{k-1})$
does not involve $\bm a_k$ nor $\bm b_k$ we have \eqref{eq:eq7}.
\end{proof}

\subsection{Arithmetical soundness}

Via a series of lemmata we shall prove Theorem \ref{theorem:hierarchyoneissound} to the effect that the hierarchy $\{\principle{R}_i\}_{i\in \omega}$ is arithmetically sound in any reasonable arithmetical theory.

\begin{theorem}\label{theorem:hierarchyoneissound}
Each of the $\principle R_i$ is arithmetically sound in any theory extending \sonetwo.
\end{theorem}

It is sufficient to prove that each of the $\principle R_{2m}$ is arithmetically sound in any reasonable arithmetical theory whence we shall focus on the principles $\rr_i$. We shall first exhibit a soundness proof of $\rr_1$ and then indicate how this is generalized to the rest of the hierarchy. And before proving $\rr_1$ we need some auxiliary lemmas.

\begin{lemma}\label{theorem:TheYformulasProveSigmaFormulas}
Let $T$ be any theory extending \sonetwo. We have that for any arithmetical sentences $E_1, A_0, B_0$ and $C_0$ that
\[
T\vdash E_1 \rhd \neg (A_0 \rhd \neg C_0) \to \existscut J \ \Box\, \big ( E_1 \to \forallcut\,  K{\in}\dot J\ \Diamond^{\dot J} (A_0 \wedge \Box^{\dot K}C_0)\big).
\]
\end{lemma}

\begin{proof}
Reason in $T$ and assume $E_1 \rhd \neg (A_0 \rhd \neg C_0)$. Note that by Lemma \ref{theorem:negatedRhdYieldsSmallWitness} we have $E_1 \rhd \forallcut K \Diamond (A_0\wedge \Box^{\dot K} C_0)$. Consequently, by Pudl\'ak's Lemma, Lemma \ref{theorem:PudlaksPrinciple}, we get $\exists J\ \big( E_1 \wedge \existscut\, K{\in} \dot J  \ \Box^{\dot J}\neg (A_0\wedge \Box^{\dot K} C_0) \rhd \bot \big)$. But this is provably the same as $\existscut J \ \Box\, \big ( E_1 \to \forallcut\,  K{\in}\dot J\ \Diamond^{\dot J} (A_0 \wedge \Box^{\dot K}C_0)\big)$ as was to be shown.
\end{proof}
\noindent

\begin{lemma}\label{theorem:EOneInterpretsAZero}
Let $T$ be any theory extending \sonetwo. We have that for any arithmetical sentences $E_1, A_0, B_0$ and $C_0$ that 
\[
T\vdash \existscut J \ \Box\, \big ( E_1 \to \forallcut\,  K{\in}\dot J\ \Diamond^{\dot J} (A_0 \wedge \Box^{\dot K}C_0)\big) \to E_1 \rhd A_0.
\]
\end{lemma}

\begin{proof}
Reason in $T$. From the assumption we get in particular that $\existscut J \ \Box\, \big ( E_1 \to \Diamond^{\dot J}A_0 \big)$ so that $\existscut J\ E_1\rhd \Diamond^{\dot J}A_0 \rhd A_0$.
\end{proof}

\begin{lemma}\label{theorem:EOneInterpretsZZero}
Let $T$ be any theory extending \sonetwo. We have that for any arithmetical sentences $E_1, A_0, B_0$ and $C_0$ that 
\[
T\vdash (A_0 \rhd B_0) \wedge \existscut J \ \Box\, \big ( E_1 \to \forallcut\,  K{\in}\dot J\ \Diamond^{\dot J} (A_0 \wedge \Box^{\dot K}C_0)\big) \to E_1 \rhd B_0 \wedge \Box C_0.
\]
\end{lemma}
\begin{proof}
Reasoning in $T$ we get from $\existscut J \ \Box\, \big ( E_1 \to \forallcut\,  K{\in}\dot J\ \Diamond^{\dot J} (A_0 \wedge \Box^{\dot K}C_0)\big)$ that $\forallcut K\ \big( E_1 \rhd A_0 \wedge \Box^{\dot K}C_0\big)$. We combine this with $A_0\rhd B_0 \to \existscut J \ \big( A_0 \wedge \Box^{\dot J}C_0 \rhd B_0 \wedge \Box C_0\big)$ to conclude $E_1 \rhd B_0 \wedge \Box C_0$.
\end{proof}
With these technical lemmas we can prove soundness of $\rr_1$.
\begin{lemma}
Let $T$ be any theory extending \sonetwo. We have that for any arithmetical sentences $E_1, A_1, B_1, A_0, B_0$ and $C_0$ that
\[
\begin{array}{lll}
T & \vdash  & A_1 \rhd B_1\wedge (A_0\rhd B_0) \to \\
 & &\ \ \ \ \neg (A_1 \rhd \neg C_1)\wedge (E_1 \rhd \neg (A_0 \rhd \neg C_0))\ \rhd \\
 & & \ \ \ \ \ \ \ \ \ \ \ \ \ \ \ \ \  B_1 \wedge (A_0\rhd B_0) \wedge \Box C_1 \wedge (E_1 \rhd A_0)\wedge (E_1\rhd B_0 \wedge \Box C_0).\\
 \end{array}
\]
\end{lemma}

\begin{proof}
We reason in $T$. Using our new technical lemma and Lemma \ref{theorem:negatedRhdYieldsSmallWitness} we get
\[
\begin{array}{lr}
\neg (A_1 \rhd \neg C_1)\wedge (E_1 \rhd \neg (A_0 \rhd \neg C_0)) & \to \\
\forallcut K \Diamond (A_1\wedge \Box^{\dot K} C_1) \ \wedge \ \existscut J \ \Box\, \big ( E_1 \to \forallcut\,  L{\in}\dot J\ \Diamond^{\dot J} (A_0 \wedge \Box^{\dot L}C_0)\big) & \to \\
\forallcut K \Diamond \Big(A_1\wedge \Box^{\dot K} C_1 \ \wedge \ \existscut \, J{\in}\dot K \ \Box^{\dot K}\, \big ( E_1 \to \forallcut\,  L{\in}\dot J\ \Diamond^{\dot J} (A_0 \wedge \Box^{\dot L}C_0)\big) \Big). & 
\end{array}
\]
The last step is due to the principle of outside-big inside-small (Lemma \ref{Lemma:OutsideBIGInsidesmall}) and allows us to conclude
\begin{align*}
\forallcut K \ \Big( \ \ \neg (A_1 & \rhd \neg C_1)\wedge  (E_1 \rhd \neg (A_0 \rhd \neg C_0)) \ \ \rhd \\
 & A_1\wedge \Box^{\dot K} C_1 \ \wedge \ \existscut \, J{\in}\dot K \ \Box^{\dot K}\, \big ( E_1 \to \forallcut\,  L{\in}\dot J\ \Diamond^{\dot J} (A_0 \wedge \Box^{\dot L}C_0)\big) \ \ \Big). \label{bla} \\ 
\end{align*}
This can be combined with the fact that  
\[
A_1 \rhd B_1\wedge (A_0\rhd B_0) \to \existscut K \ \big( A_1 \wedge \sigma^{\dot K}\rhd B_1\wedge \sigma \wedge (A_0\rhd B_0) \ \big)
\]
for this particular $K$ holds for any $\sigma \in \Sigma_1$ to conclude
\begin{align*}
A_1 \rhd B_1\wedge & (A_0\rhd B_0) \to  \neg (A_1  \rhd \neg C_1)\wedge  (E_1 \rhd \neg (A_0 \rhd \neg C_0)) \ \ \rhd \\
 &  B_1\wedge (A_0\rhd B_0)\wedge \Box C_1 \ \wedge \ \existscut  J \ \Box\, \big ( E_1 \to \forallcut\,  L{\in}\dot J\ \Diamond^{\dot J} (A_0 \wedge \Box^{\dot L}C_0)\big).
\end{align*}
(Note that $\Box^{\dot K} C_1 \, \wedge \, \existscut \, J{\in}\dot K \ \Box^{\dot K}\, \big ( E_1 \to \forallcut\,  L{\in}\dot J\ \Diamond^{\dot J} (A_0 \wedge \Box^{\dot L}C_0)\big)$ is equivalent to an $\exists \Sigma_1^b$ sentence relativized to $\dot K$.) Our technical lemmas \ref{theorem:EOneInterpretsAZero} and \ref{theorem:EOneInterpretsZZero} tell us that 
\begin{align*}
(A_0\rhd B_0)\wedge   \existscut  J \ \Box\, \big ( E_1 \to \forallcut\,  L{\in}\dot J\ \Diamond^{\dot J} (A_0 \wedge \Box^{\dot L}C_0)\big) \ \ \to \\
\ \ \ \ (E_1 \rhd A_0) \wedge (E_1 \rhd B_0\wedge \Box C_0)
\end{align*} 
and we are done.
\end{proof}
The soundness proofs for $\rr_k$ is essentially not much different. We shall indicate where the soundness proof for $\rr_1$ needs to be modified and begin with modifications of the technical lemmas. 

However, first we must inductively define a series of important formulas. In our definition we work with more variables than actually needed. However, we have chosen to do so since our variables can be interpreted as numbers or as formulas and we wish to avoid expressions like $\forallcut J \ \Box\, \exists J{\in} \dot J \ \phi$.
\begin{align*}
\mathcal H_1 & := \existscut J_1\ \Box \big( E_1 \to \forallcut K_1 {\in}\dot J_1\ \Diamond^{\dot J_1} (A_0 \wedge \Box^{\dot K_1}C_0)\big);\\
\mathcal H_{k+1} & := \existscut J_{k+1}\ \Box \big( E_{k+1} \to \forallcut K_{k+1}{\in}\dot J_{k+1} \ \Diamond^{\dot J_{k+1}} (A_k \wedge \Box^{\dot K_{k+1}}C_k \wedge \mathcal H_k^{\dot K_{k+1}})\big).\\
\end{align*}
It is easy to see that for each $k>0$ the formula $\mathcal H_k$ is an $\exists \Sigma_1^b$ formula. The next lemmas show us that $\mathcal H_{k+1}$ are $\exists \Sigma_1^b$ consequences of the $\Sigma_3$ statements $E_{k+1}\rhd \yy_k$ which contain all the essential information for proving soundness. 
First we prove a simple modification of Lemma \ref{theorem:TheYformulasProveSigmaFormulas}.

\begin{lemma}\label{theorem:TheYformulasProveSigmaFormulasGeneralized}
Let $T$ be any theory extending \sonetwo. We have that for any arithmetical sentences $E_1, A_0, B_0$ and $C_0$  and for any $\exists \Sigma^b_1$ formula $\sigma$ that
\[
T\vdash E_1 \rhd \neg (A_0 \rhd \neg C_0) \wedge \sigma \to \existscut J \ \Box\, \big ( E_1 \to \forallcut\,  K{\in}\dot J\ \Diamond^{\dot J} (A_0 \wedge \Box^{\dot K}C_0 \wedge \sigma^{\dot K})\big).
\]
\end{lemma}

\begin{proof}
We repeat the proof of Lemma \ref{theorem:TheYformulasProveSigmaFormulas}. Note that, by our reading conventions the antecedent $E_1 \rhd \neg (A_0 \rhd \neg C_0) \wedge \sigma$ should be read as $E_1 \rhd \big( \neg (A_0 \rhd \neg C_0) \wedge \sigma\big)$. We reason in $T$ and see that 
\[
\begin{array}{lll}
\neg (A_0 \rhd \neg C_0) \wedge \sigma  &\to& \forallcut K \Diamond (A_0\wedge \Box^{\dot K} C_0) \wedge \sigma\\
&\to & \forallcut K \Diamond (A_0\wedge \Box^{\dot K} C_0 \wedge \sigma^{\dot K}). 
\end{array}
\]
As before, the latter implies $\existscut J \ \Box\, \big ( E_1 \to \forallcut\,  K{\in}\dot J\ \Diamond^{\dot J} (A_0 \wedge \Box^{\dot K}C_0 \wedge \sigma^{\dot K})\big)$.
\end{proof}

With this lemma we see that the $\mathcal H_{k+1}$ are an $\exists \Sigma_1^b$ encoding of information present in $E_{k+1}\rhd \yy_k$:

\begin{lemma}\label{theorem:HisImplied}
Let $T$ be a theory containing \sonetwo and let the formulas $E_i, A_i$, and $C_i$ be arbitrary. For any number $k$ we have that 
\[
T\vdash E_{k+1}\rhd \yy_k \ \to \ \mathcal H_{k+1}.
\]
\begin{proof}
By an external induction on $k$. For $k=0$ this is simply Lemma \ref{theorem:TheYformulasProveSigmaFormulas}. For the inductive case we reason in $T$ and see that $E_{k+2} \rhd \yy_{k+1} \equiv E_{k+2} \rhd \neg (A_{k+1} \rhd \neg C_{k+1})\wedge (E_{k+1}\rhd \yy_k)$. By the inductive hypothesis we have that $E_{k+1}\rhd \yy_k \to \mathcal H_{k+1}$ so that 
$E_{k+2} \rhd \yy_{k+1} \to E_{k+2}\rhd \neg (A_{k+1} \rhd \neg C_{k+1}) \wedge \mathcal H_{k+1}$. Since $\mathcal H_{k+1}$ is equivalent to an $\exists\Sigma_1^b$ formula, by Lemma \ref{theorem:TheYformulasProveSigmaFormulasGeneralized} we see that 
\[
E_{k+2}\rhd \neg (A_{k+1} \rhd \neg C_{k+1}) \wedge \mathcal H_{k+1} \to \mathcal H_{k+2}
\]
as was to be shown.
\end{proof}
\end{lemma}
Moreover, the $\mathcal H_{k+1}$ formulas contain all the information to get the induction going as shown by the following lemma.

\begin{lemma}\label{theorem:InductiveStepInRRsoundnessProof}
Let $T$ be a theory containing \sonetwo and let the formulas $E_i, A_i, B_i$, and $C_i$ be arbitrary. For any number $k$ we have that 
\[
T\vdash (\xx_k) \, \wedge \, \mathcal H_{k+1} \ \to \ E_{k+1} \rhd \zz_k . 
\]
\end{lemma}

\begin{proof}
By induction on $k$ where the case $k=0$ is just lemma \ref{theorem:EOneInterpretsZZero}. For the inductive case, we reason in $T$ and assume $(\xx_{k+1}) \, \wedge \, \mathcal H_{k+2}$.

From the definition of $\mathcal H_{k+2}$ we get 
\[
\existscut J_{k+2}\ \Box \big( E_{k+2} \to \forallcut K_{k+2}{\in}\dot J_{k+2} \ \Diamond^{\dot J_{k+2}} (A_{k+1} \wedge \Box^{\dot K_{k+2}}C_{k+1} \wedge \mathcal H_{k+1}^{\dot K_{k+2}})\big)
\] 
\[
\mbox{so that }\existscut J_{k+2}\, \forallcut K_{k+2}\ \Box \big( E_{k+2} \to  \Diamond^{\dot J_{k+2}} (A_{k+1} \wedge \Box^{\dot K_{k+2}}C_{k+1} \wedge \mathcal H_{k+1}^{\dot K_{k+2}})\big)
\] 
whence
\begin{equation}\label{equation:InductiveStepInRRsoundnessProof}
\forallcut K_{k+2}\ \big( E_{k+2} \rhd A_{k+1} \wedge \Box^{\dot K_{k+2}}C_{k+1} \wedge \mathcal H_{k+1}^{\dot K_{k+2}} \, \big).
\end{equation}

From $\xx_{k+1}$ --which is by definition equal to $A_{k+1}\rhd B_{k+1} \wedge (\xx_k)$-- we find via Pudl\'ak's lemma, Lemma \ref{theorem:PudlaksPrinciple}, a specific cut $\overline K_{k+2}$ such that for any formula $\sigma$ in $\Sigma_1$ we obtain $A_{k+1} \wedge \sigma^{\overline K_{k+2}}\rhd B_{k+1} \wedge (\xx_k) \wedge \sigma$. We can plug in this cut $\overline K_{k+2}$ to \eqref{equation:InductiveStepInRRsoundnessProof} to obtain via transitivity of $\rhd$ that 
\[
E_{k+2} \rhd B_{k+1} \wedge (\xx_k) \wedge \Box C_{k+1} \wedge \mathcal H_{k+1}.
\]
We are almost done but $B_{k+1} \wedge (\xx_k) \wedge \Box C_{k+1} \wedge \mathcal H_{k+1}$ is not quite equal to $\zz_{k+1}$ as was needed. The missing conjuncts are $E_{k+1} \rhd A_k$ and $E_{k+1} \rhd \zz_k$. The first is easily seen to follow from $\mathcal H_{k+1}$ and the second follows from the inductive hypothesis applied to $(\xx_k) \wedge \mathcal H_{k+1}$.
\end{proof}
We are now ready to prove Theorem \ref{theorem:hierarchyoneissound} that the whole hierarchy is arithmetically sound.

\begin{theorem}
Let $T$ be a theory containing \sonetwo and let $A_i,B_i,C_i$ and $E_i$ be arbitrary arithmetical formulas. We have for each number $k$ that 
\[
T\vdash \rr_k \ \ \mbox{id est }\ \ T\vdash \xx_k \to \yy_k \rhd \zz_k.
\]
\end{theorem}

\begin{proof}
By an external induction on $k$ where the base case is the soundness of $\rr_0$ which has been proven in \cite{GorisJoosten:2011:ANewPrinciple}. Thus, we reason in $T$ assuming $A_{k+1}\rhd B_{k+1} \wedge (\xx_{k})$. We need to conclude that $\yy_{k+1}\rhd \zz_{k+1}$. But $\yy_{k+1}$ is nothing but $\neg (A_{k+1}\rhd \neg C_{k+1}) \wedge (E_{k+1} \rhd \yy_{k})$. By Lemma 
%\ref{theorem:TheYformulasProveSigmaFormulas} 
\ref{theorem:HisImplied} 
we know that $(E_{k+1} \rhd \yy_{k}) \to \mathcal H_{k+1}$. Using this and reasoning as before we obtain
\[
\begin{array}{lll}
\neg (A_{k+1} \rhd \neg C_{k+1}) \wedge (E_{k+1} \rhd \yy_{k})  &\to& \forallcut K \Diamond (A_{k+1}\wedge \Box^{\dot K} C_{k+1}) \wedge (E_{k+1} \rhd \yy_{k})\\
  &\to& \forallcut K \Diamond (A_{k+1}\wedge \Box^{\dot K} C_{k+1}) \wedge \mathcal H_{k+1}\\
&\to & \forallcut K \Diamond (A_{k+1}\wedge \Box^{\dot K} C_{k+1} \wedge \mathcal H_{k+1}^{\dot K}). 
\end{array}
\]
Consequently, 
\[
\forallcut K \ \big(\neg (A_{k+1} \rhd \neg C_{k+1}) \wedge (E_{k+1} \rhd \yy_{k})\rhd A_{k+1}\wedge \Box^{\dot K} C_{k+1} \wedge \mathcal H_{k+1}^{\dot K} \big).
\]
This can be combined with Pudl\'ak's Lemma on $A_{k+1}\rhd B_{k+1} \wedge (\xx_k)$ to obtain
\[
\neg (A_{k+1} \rhd \neg C_{k+1}) \wedge (E_{k+1} \rhd \yy_{k})\rhd B_{k+1} \wedge (\xx_k) \wedge \Box C_{k+1} \wedge \mathcal H_{k+1}. 
\]
It is easy to see that $\mathcal H_{k+1}$ implies $E_{k+1}\rhd A_k$. Moreover, Lemma \ref{theorem:InductiveStepInRRsoundnessProof} tells us that $(\xx_k)\wedge \mathcal H_{k+1} \to E_{k+1} \rhd \zz_k$ so that we may conclude
\[
\neg (A_{k+1} \rhd \neg C_{k+1}) \wedge (E_{k+1} \rhd \yy_{k})\rhd B_{k+1} \wedge (\xx_k) \wedge \Box C_{k+1} \wedge (E_{k+1}\rhd A_k ) \wedge (E_{k+1} \rhd \zz_k)
\]
as was to be shown.
\end{proof}

\section{A broad series of principles}

In this section we present a different series of principles. We refer to this series as the broad series since the frame-conditions  --see Figure \ref{figure:broadSeries}-- are typically represented over a broader area than the slim hierarchy as discussed above.

\subsection{A broad series}
In order to define the second series we first define a series of auxiliary formulas. For any $n\geq 1$ we define the schemata $\uu_n$ as follows.
\begin{align*}
\uu_1 &:= \Diamond\neg(D_1\rhd \neg C),\\
\uu_{n+2} &:= \Diamond((D_{n+1}\rhd D_{n+2})\wedge\uu_{n+1}).
\end{align*}
Now, for $n\geq 0$ we define the schemata $\principle R^n$ as follows.
\begin{align*}
\principle{R}^0 &:= A\rhd B\rightarrow\neg(A\rhd \neg C)\rhd B\wedge\Box C,\\
\principle{R}^{n+1} &: = A\rhd B\rightarrow\uu_{n+1}\wedge(D_{n+1}\rhd A)\rhd B\wedge\Box C.
\end{align*}

As an illustration we shall calculate the first four principles.
\[
\begin{array}{lll}
\principle{R}^0 & :=  & A \rhd B \to \neg (A \rhd \neg C) \rhd B \wedge \Box C\\
\principle{R}^1 &:=& A \rhd B \to \Diamond \neg(D_1 \rhd \neg C) \wedge (D_1 \rhd A)  \rhd B \wedge \Box  C\\
\principle{R}^2 &:=& A \rhd B \to  \Diamond\Big[ (D_1 \rhd  D_2) \wedge\Diamond\neg(D_1 \rhd \neg C)\Big] \wedge (D_2 \rhd A) \rhd B \wedge \Box  C\\

\principle{R}^3 &:=& A \rhd B \to  \Diamond \Big( (D_2\rhd D_3) \wedge \Diamond\Big[ (D_1 \rhd  D_2) \wedge\Diamond\neg(D_1 \rhd \neg C)\Big] \Big) \wedge (D_3 \rhd A) \\
 & & \ \ \ \ \ \ \ \ \ \ \ \ \ \ \ \ \ \ \ \ \ \ \ \ \ \ \ \ \ \ \ \ \ \ \ \ \ \ \ \  \ \ \ \ \ \ \ \ \ \ \ \ \ \ \ \ \ \ \ \ \ \ \ \ \ \ \ \ \ \ \ \ \ \ \ \ \ \  \rhd B \wedge \Box  C

\end{array}
\]
While the series $\principle{R}_i$ did define a hierarchy in that $\principle{R}_{i+1} \vdash \principle R_i$, we shall see that no such relation holds for the series $\principle{R}^i$.

\subsection{Frame conditions}

It is not hard to determine the frame condition for the first couple of principles in this series and in Figure \ref{figure:broadSeries} we have depicted the first three frame-conditions. In this section we shall prove that the correspondence proceeds as expected. Informally, the frame condition for $\principle{R}^n$ shall be the universal closure of
\begin{equation}\label{equation:informalFrameConditionForBroadHierarchy}
x_{n+1}Rx_n \ldots Rx_0 Ry_0 S_{x_1} y_1 \ldots S_{x_{n}}y_{n} S_{x_{n+1}}y_{n+1} R u \to y_0 S_{x_0}u.
\end{equation}

%\vspace{-1.5cm}
\begin{figure}[h]
\begin{picture}(105,90)(0,0)
\unitlength=1,5mm
%\unitlength=2mm
\put(2,-20){
%\unitlength=2mm
\gasset{ExtNL=y,NLdist=0.2,Nw=1,Nh=1}
\node[NLangle=0](X1)(0,20){$x_1$}
\node[NLangle=0](X0)(0,30){$x_0$}
\node[NLangle=0](Y0)(0,40){$y_0$}
\node[NLangle=0](Y1)(10,40){$y_1$}
\node[NLangle=0](Z)(10,50){$z$}
\drawedge(X1,X0){}
\drawedge(X0,Y0){}
\drawedge(X1,Y1){}
\drawedge(Y1,Z){}
\drawbpedge[ELside=r,ELpos=60,ELdist=0.2](Y0,45,5,Y1,225,5){$S_{x_1}$}
\drawbpedge[ELpos=30,ELdist=0.2,dash={1.5}0](Y0,90,5,Z,270,5){$S_{x_0}$}
}

\put(22,-10){
%\unitlength=2mm
\gasset{ExtNL=y,NLdist=0.2,Nw=1,Nh=1}
\node[NLangle=0](X2)(0,10){$x_2$}
\node[NLangle=0](X1)(0,20){$x_1$}
\node[NLangle=0](X0)(0,30){$x_0$}
\node[NLangle=0](Y0)(0,40){$y_0$}
\node[NLangle=0](Y1)(10,40){$y_1$}
\node[NLangle=0](Y2)(20,40){$y_2$}
\node[NLangle=0](Z)(20,50){$z$}
\drawedge(X2,X1){}
\drawedge(X1,X0){}
\drawedge(X0,Y0){}
\drawedge(X1,Y1){}
\drawedge(X2,Y2){}
\drawedge(Y2,Z){}
\drawbpedge[ELside=r,ELpos=60,ELdist=0.2](Y0,45,5,Y1,225,5){$S_{x_1}$}
\drawbpedge[ELside=r,ELpos=60,ELdist=0.2](Y1,45,5,Y2,225,5){$S_{x_2}$}
\drawbpedge[ELpos=30,ELdist=0.2,dash={1.5}0](Y0,60,8,Z,250,8){$S_{x_0}$}
}

\put(52,0){
%\unitlength=2mm
\gasset{ExtNL=y,NLdist=0.2,Nw=1,Nh=1}
\node[NLangle=0](X3)(0,0){$x_3$}
\node[NLangle=0](X2)(0,10){$x_2$}
\node[NLangle=0](X1)(0,20){$x_1$}
\node[NLangle=0](X0)(0,30){$x_0$}
\node[NLangle=0](Y0)(0,40){$y_0$}
\node[NLangle=0](Y1)(10,40){$y_1$}
\node[NLangle=0](Y2)(20,40){$y_2$}
\node[NLangle=0](Y3)(30,40){$y_3$}
\node[NLangle=0](Z)(30,50){$z$}
\drawedge(X3,X2){}
\drawedge(X2,X1){}
\drawedge(X1,X0){}
\drawedge(X0,Y0){}
\drawedge(X1,Y1){}
\drawedge(X2,Y2){}
\drawedge(X3,Y3){}
\drawedge(Y3,Z){}
\drawbpedge[ELside=r,ELpos=60,ELdist=0.2](Y0,45,5,Y1,225,5){$S_{x_1}$}
\drawbpedge[ELside=r,ELpos=60,ELdist=0.2](Y1,45,5,Y2,225,5){$S_{x_2}$}
\drawbpedge[ELside=r,ELpos=60,ELdist=0.2](Y2,45,5,Y3,225,5){$S_{x_3}$}
\drawbpedge[ELpos=30,ELdist=0.2,dash={1.5}0](Y0,55,13,Z,245,13){$S_{x_0}$}
}

\end{picture}
\caption{From left to right, this figure depicts the frame conditions $\mathcal F^0$ through $\mathcal F^2$ corresponding to $\principle{R}^0$ through $\principle{R}^2$. The reading convention is as always: if all the un-dashed relations are present as in the picture, then also the dashed relation should be there.}\label{figure:broadSeries}
\end{figure}
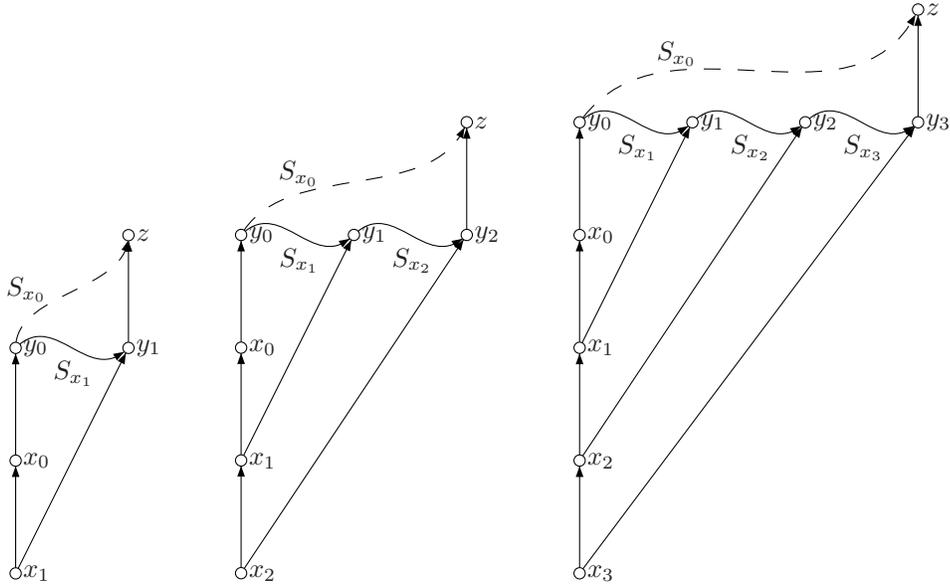

In order to make this frame condition precise and prove it, we shall first recast it in a recursive fashion. In writing \eqref{equation:informalFrameConditionForBroadHierarchy} recursively we shall use those variables that will emphasize the relation with \eqref{equation:informalFrameConditionForBroadHierarchy}. Of course, free variables can be renamed at the readers liking.

First, we start by introducing a relation $\mathcal B_n$ that captures the antecedent of \eqref{equation:informalFrameConditionForBroadHierarchy}. Note that this antecedent says that first there is a chain of points $x_i$ related by $R$, followed by a chain of points $y_i$ related by different $S$ relations. The relation $\mathcal B_n$ will be applied to the end-points of both chains where the condition on the intermediate points is imposed by recursion.

\begin{align*}
\mathcal B_0(x_1,x_0,y_0,y_1) &:= x_1Rx_1Ry_0S_{x_1}y_1,\\
\mathcal B_{n+1}(x_{n+2},x_{0},y_0,y_{n+2}) &:= \exists x_{n+1}, y_{n+1} \big(x_{n+2}Rx_{n+1}\wedge \mathcal B_n(x_{n+1},x_0,y_0,y_{n+1}) \\
 & \ \ \ \ \ \ \ \ \ \ \  \ \ \ \ \ \ \ \ \ \ \  \ \ \ \ \ \ \ \ \ \ \  \ \ \ \ \ \ \ \ \ \ \ \wedge y_{n+1}S_{x_{n+2}}y_{n+2} \big).
\end{align*}
For every $n\geq 0$ we can now define the first order frame condition $\mathcal F^n$ as follows.
\[
\mathcal F^n := \forall x_{n+1},x_0,y_0,y_{n+1} \ \big(\mathcal B_n(x_{n+1},x_0,y_0,y_{n+1})\Rightarrow \forall u\, (y_{n+1}Ru\Rightarrow y_0S_{x_0}u)\big)
%\enspace.
.
\]

Sometimes we shall write $x_{n+1}\mathcal B_n[x_0,y_0] \, y_{n+1}$ conceiving the quaternary relation $\mathcal B_n$ as a binary relation indexed by the pair $x_0,y_0$. In what follows we let $F=\langle W,R,S\rangle$ be an arbitrary Veltman-frame. The next lemma follows from an easy induction on $n$.
\begin{lemma}\label{theorem:HSubseteqR}
For each number $n$ we have that $\mathcal B_n[x_0,y_0] \, \subseteq \, R$, that is, if $x_{n+1}\mathcal B_n[x_0,y_0]\, y_{n+1}$, then $x_{n+1}Ry_{n+1}$.
\end{lemma}
To prove that $F\models \mathcal F^n$ implies $F\models \principle{R}^n$ we first need a technical lemma.
\begin{lemma}\label{lemm:lemm2}
Let $w\in W$ and $\Vdash$ be a forcing relation on $F$.
If 
\[
x_{k+1}\Vdash \uu_{k+1}\wedge(D_{k+1}\rhd A),
\]
then there exist $x_0$, $y_0$ and $y_{k+1}$ such that
$\mathcal B_k(x_{k+1},x_0,y_0,y_{k+1})$, $x_0R^C_\Vdash y_0$ and $y_{k+1}\Vdash A$.
\end{lemma}
\begin{proof}
Induction on $k$.
If $k=0$ then $\uu_{k+1}=\Diamond\neg(D_1\rhd \neg C)$ and the statement is easily checked.
For the inductive case, we assume 
\[
x_{k+2}\Vdash \uu_{k+2} \wedge (D_{k+2} \rhd A).
\] 
Recall that $\uu_{k+2}:=\Diamond((D_{k+1}\rhd D_{k+2})\wedge\uu_{k+1})$.
Thus, there exists some
$x_{k+1}$ with $x_{k+2}Rx_{k+1}$ and 
\[
x_{k+1}\Vdash(D_{k+1}\rhd D_{k+2})\wedge\uu_{k+1}.
\]
Applying the (IH) (with $D_{k+2}$ substituted for $A$)
we find $x_0$, $y_0$ and $y_{k+1}$ with $\mathcal B_k(x_{k+1},x_0,y_0,y_{k+1})$, $x_0R^C_\Vdash y_0$ and $y_{k+1}\Vdash D_{k+2}$.
As $\mathcal B_k(x_{k+1},x_0,y_0,y_{k+1})$ we get $x_{k+1}Ry_{k+1}$ (Lemma \ref{theorem:HSubseteqR}). Since we had $x_{k+2}Rx_{k+1}$ we see that $x_{k+2}Ry_{k+1}\Vdash D_{k+2}$, and since $x_{k+2}\Vdash D_{k+2}\rhd A$, we find some $y_{k+2}$ with $y_{k+1}S_{x_{k+2}}y_{k+2}$ and $y_{k+2} \Vdash A$. By definition of $\mathcal B_{k+1}$ we have $\mathcal B_{k+1}(x_{k+2},x_0,y_0,y_{k+2})$.
\end{proof}

\begin{corollary}\label{corollary:FrameConditionImpliesPrincipleBroadHierarchy}
If $F\models \mathcal F^n$ then $F\models \principle{R}^n$.
\end{corollary}
\begin{proof}
Induction on $n$. For $n=0$ this is known (see \cite{GorisJoosten:2011:ANewPrinciple}), so we assume $n>0$. Let $\Vdash$ be a forcing relation,
let $x_{n+1},x_n\in W$ and assume $x_{n+1}\Vdash A\rhd B$, $x_{n+1}Rx_n$ and
$x_n\Vdash \uu_n\wedge(D_{n}\rhd A)$.
By Lemma \ref{lemm:lemm2} we find $x_0$, $y_0$ and $y_n$ such that
$\mathcal B_{n-1}(x_n,x_0,y_0,y_n)$, with $x_0R^C_\Vdash y_0$ and $y_n\Vdash A$.
We have that $\mathcal B_{n-1}(x_n,x_0,y_0,y_n)$ implies $x_nRy_n$ (Lemma \ref{theorem:HSubseteqR}) and thus since $x_{n+1}Rx_n$ we also we have $x_{n+1}Ry_n\Vdash A$.
By assumption $x_{n+1}\vdash A\rhd B$ so that for some $y_{n+1}$ we have $y_nS_{x_{n+1}}y_{n+1} \Vdash B$.
Clearly, we also have $x_{n}S_{x_{n+1}}y_{n+1}$ so that we are done if we have shown that $y_{n+1}\Vdash\Box C$.
To this extent, we choose some $u$ with $y_{n+1}Ru$.
Since we have that $\mathcal B_n(x_{n+1},x_0,y_0,y_{n+1})$, by $\mathcal F^n$ we have also $y_0S_{x_0}u$. But $x_0R^C_\Vdash y_0$ and thus we have $u\Vdash C$, as required.
\end{proof}

To prove the converse implication, we start again with a technical lemma. As before we shall denote by $\bm a$, ${\bm b}$, ${\bm c}$, and ${\bm d}_k$, propositional variables that shall play the role of the $A$, $B$, $C$ and $D_k$ respectively in the principles $\principle{R}^n$. Let $\uup_k$ denote the formula that arises by simultaneously substituting $\bm c$ for $C$ and ${\bm d}_k$ for $D_k$ in $\uu_k$.

\begin{lemma}\label{lemm:lemm3}
Let $\{ \bm a, \bm c, \bm d_1, \ldots , \bm d_{k+1}\}$ be a collection of distinct propositional variables. If $F \models \mathcal B_k(x_{k+1},x_0,y_0,y_{k+1})$, then there exists a forcing relation $\Vdash$ on $F$ such that
\begin{enumerate}
\item
$x_{k+1}\Vdash\uup_{k+1}\wedge(\bm d_{k+1}\rhd \bm a)$;
\item
$x\Vdash \bm c$ iff $y_0S_{x_0}x$;
\item
$x\Vdash \bm a \Leftrightarrow x=y_{k+1}$;
\item
$x\nVdash \bm p$ for any $\bm p \notin \{  \bm d_1, \ldots , \bm d_{k+1}, \bm c, \bm a\}$.
\end{enumerate}
\end{lemma}
\begin{proof}
The idea is very simple using the informal description of $\mathcal B_k$ being the antecedent of \eqref{equation:informalFrameConditionForBroadHierarchy}. We define a valuation $\Vdash$ so that $\bm d_{i+1}$ is only true at $y_i$ and $\bm a$ is only true at $y_{k+1}$. Moreover, we define $x\Vdash \bm c$ iff $y_0S_{x_0}x$ and $x\nVdash \bm p$ for any $\bm p \notin \{  \bm d_1, \ldots , \bm d_{k+1}, \bm c, \bm a\}$. It is not hard to see that $x_{k+1}\Vdash\uup_{k+1}\wedge(\bm d_{k+1}\rhd \bm a)$ for this valuation $\Vdash$. 

To make the argument precise, we proceed by induction on $k$. If $k=0$ then $\mathcal B_k(x_1,x_0,y_0,y_1)$ simply means $x_1Rx_0Ry_0S_{x_1}y_1$ and we define
\[
x\Vdash \bm a \Leftrightarrow x=y_1, \quad x\Vdash \bm c \Leftrightarrow y_0S_{x_0}x\quad\textrm{and,}\quad
x\Vdash \bm d_1\Leftrightarrow x=y_0.
\]
The lemma is easily checked if we further define $x\nVdash \bm p$ for any $\bm p \notin \{  \bm d_1, \bm c, \bm a\}$.

For the inductive case we consider $k>0$. Then $\mathcal B_k(x_{k+1},x_0,y_0,y_{k+1})$ implies that there are $x_k$ and $y_k$ such that 
\[
x_{k+1} R\, x_k\, \mathcal B_{k-1}[x_0,y_0]\,y_{k} \, S_{x_{k+1}}\, y_{k+1}.
\] 
The (IH) (with $\bm d_{k+1}$ substituted for $\bm a$) gives a forcing relation $\Vdash$ such that
\[
x_k\Vdash\uup_k\wedge (\bm d_k\rhd \bm d_{k+1}),\quad x_0R^{\bm c}_\Vdash y_0, \quad x\Vdash \bm d_{k+1}\Leftrightarrow x=y_{k}
\]
and $x\nVdash \bm p \textrm{ for } \bm p \notin \{  \bm d_1, \ldots , \bm d_{k+1}, \bm c\}$.
So we have $x_{k+1}\Vdash\Diamond\big(\uup_k\wedge(\bm d_k\rhd \bm d_{k+1})\big)$; in other words $x_{k+1}\Vdash \uup_{k+1}$.
We now define $\Vdash'$ as follows
\[
x\Vdash' \bm a \Leftrightarrow x=y_{k+1} \quad \text{and} \quad x \Vdash' \bm p \Leftrightarrow x\Vdash \bm p \textrm{ for } \bm p \neq \bm a.
\]
Clearly, the properties $x_k\Vdash\uup_k\wedge (\bm d_k\rhd \bm d_{k+1})$, $aR^{\bm c}_\Vdash b$, $x\Vdash \bm d_{k+1}\Leftrightarrow x=y_{k}$ simply extend to $\Vdash'$ and likewise we have that $x\nVdash' \bm p$ for any $\bm p \notin \{  \bm d_1, \ldots , \bm d_{k+1}, \bm c, \bm a\}$.
Moreover, we now have $x_{k+1}\Vdash' \bm d_{k+1}\rhd \bm a$ as well.
\end{proof}

As a corollary to this lemma, we can now obtain the full the frame conditions for the principles $\principle{R}^n$.

\begin{theorem}
For each number $n$ we have $F\models \mathcal F^n$ iff $F\models \principle{R}^n$.
\end{theorem}

\begin{proof}
The $\Rightarrow$ direction is just Corollary \ref{corollary:FrameConditionImpliesPrincipleBroadHierarchy} so we focus on the other direction. Thus, we suppose that $F\models \principle{R}^n$, consider any $x_{n+1},x_0,y_0,y_{n+1}\in W$ with $\mathcal B_n(x_{n+1},x_0,y_0,y_{n+1})$ and set out to show that for any $u$ with $y_{n+1}Ru$ we have $y_0S_{x_0}u$. 
We now apply Lemma \ref{lemm:lemm3} and simultaneously substitute $\bm a$ for $\bm d_{n+1}$ and $\bm b$ for $\bm a$ to see that there exists a forcing relation $\Vdash$ such that
\[
x_{n+1}\Vdash\uup_{n+1} [\bm d_{n+1}/\bm a] \wedge(\bm a\rhd \bm b),\quad x\Vdash \bm c\Leftrightarrow y_0S_{x_0}x
\quad\textrm{and}\quad x\Vdash \bm b \Leftrightarrow x=y_{n+1}.
\]
Since $n=0$ is known, we assume $n>0$. Thus, we find $x_n$ with $x_{n+1}Rx_n$ and $x_n\Vdash  \uup_{n-1} \wedge \bm d_{n}\rhd \bm a$ (note that $\uup_{n-1}[\bm d_{n+1}/\bm a] = \uup_{n-1}$).
Using $F\models  \overline{\principle R}^n$ we see that there must exist some
$x$ with $x\Vdash \bm b \wedge\Box \bm c$. But $y_{n+1}$ is the only world that forces $\bm b$ thus necessarily $y_{n+1}\Vdash\Box \bm c$. By the choice of $\Vdash$ we thus have that if $y_{n+1}Ru$ then $y_0S_{x_0}u$.
\end{proof}

Using the frame condition we readily see that the broad series of principles does not define a hierarchy.

\begin{corollary}
For $n\neq m$ we have $\extil{R}^{n}\nvdash \extil{R}^m$.
\end{corollary}

\begin{proof}
For each $m\neq n$ it is easy to exhibit a frame $F$ so that $F\models \mathcal F^n$ but $F\not \models \mathcal F^m$.
\end{proof}

\subsection{Arithmetical soundness}
We will now see that all the principles $\principle R^n$ are arithmetically sound and begin with a simple lemma.

\begin{lemma}\label{theorem:LemmaBroadHierarchyIsSound}
For any theory $T$ extending \sonetwo and any natural number $n>0$, we have that
\[
T\vdash \uu_n  \to \forallcut K \, \Diamond (D_n \wedge \Box^{\dot K} C).
\]
\end{lemma}

\begin{proof}
We proceed by induction on $n$ and first consider $n=1$. Thus, we reason in $T$ and assume $\uu_1$, that is, $\Diamond \neg (D_1\rhd \neg C)$. We conclude $\Diamond \forallcut K \, \Diamond (D_1 \wedge \Box^{\dot K} C)$, whence $\forallcut K \, \Diamond \Diamond (D_1 \wedge \Box^{\dot K} C)$ and also $\forallcut K\, \Diamond (D_1 \wedge \Box^{\dot K} C)$ as was to be shown.

Next, we consider the inductive case, again reasoning in $T$ and assuming $\uu_{n+1}$ which is $\Diamond \big( (D_n\rhd D_{n+1}) \wedge \uu_n \big)$. By the (IH) we conclude from $\uu_n$ that 
\begin{equation}\label{equation:IHForBroadHierarchyLemma}
\forallcut J \, \Diamond\, (D_n \wedge \Box^{\dot J} C).
\end{equation} 
By Lemma \ref{theorem:KeepWitnessSmall} we obtain from $D_n \rhd D_{n+1}$ that 
\begin{equation}\label{eq:BroadHGeneralizedPudlak}
\forallcut K \, \existscut J\ D_n\wedge \Box^{\dot J}C \rhd D_{n+1} \wedge \Box^{\dot K} C.
\end{equation} 
Combining $D_n\wedge \Box^{\dot J}C \rhd D_{n+1} \wedge \Box^{\dot K} C \to \big( \Diamond (D_n\wedge \Box^{\dot J}C) \to \Diamond (D_{n+1} \wedge \Box^{\dot K} C) \big)$ with \eqref{equation:IHForBroadHierarchyLemma} and \eqref{eq:BroadHGeneralizedPudlak} under a $\Diamond$ we conclude that 
\[
\begin{array}{lll}
\Diamond \big( (D_n\rhd D_{n+1}) \wedge \uu_n \big) &\to & \Diamond \big(  \forallcut K \, \Diamond \, (D_{n+1} \wedge \Box^{\dot K} C)\big )\\
&\to & \forallcut K \, \Diamond \big( \Diamond \, (D_{n+1} \wedge \Box^{\dot K} C)\big )\\
&\to & \forallcut K \,  \Diamond \, (D_{n+1} \wedge \Box^{\dot K} C)\\
\end{array}
\]
as was to be shown.
\end{proof}

With this lemma, we can now prove the soundness of the series $\principle R^n$.

\begin{theorem}
For each natural number $n$ we have that $\principle R^n$ is arithmetically sound in any theory $T$ extending \sonetwo.
\end{theorem}

\begin{proof}
Since we already know that ${\sf R}^0$ is sound, we consider $n>0$. We reason in $T$, assume $A\rhd B$ and set out to prove $\uu_n \wedge (D_n\rhd A) \rhd B\wedge \Box C$. By Pudl\'ak's Lemma we get
\begin{equation}\label{equation:SoundnessProof:Pudlak}
\existscut J\ A\wedge \Box^JC \rhd B \wedge \Box C.
\end{equation}
On the other hand, by the generalization of Pudl\'ak's Lemma (Lemma \ref{theorem:KeepWitnessSmall}) applied to $D_n \rhd A$ we obtain that $\forallcut J \, \existscut K \ D_n \wedge \Box^{\dot K}C \rhd A \wedge \Box^{\dot J} C$ so that $\forallcut J \, \existscut K \ \big( \Diamond (D_n \wedge \Box^{\dot K}C )\to \Diamond ( A \wedge \Box^{\dot J} C) \big)$. By Lemma \ref{theorem:LemmaBroadHierarchyIsSound} we see that $\uu_n \to \forallcut K \Diamond (A \wedge \Box^{\dot K} C)$.
Combining these last two observations, we see that $\uu_n \wedge (D_n\rhd A) \to \forallcut J\, \Diamond \, (A \wedge \Box^{\dot J} C)$ so that $\forallcut J\ \uu_n \wedge (D_n\rhd A) \rhd A \wedge \Box^{\dot J} C$. Combining this with \eqref{equation:SoundnessProof:Pudlak} yields $\uu_n \wedge (D_n\rhd A) \rhd B\wedge \Box C$ as was to be shown.
\end{proof}

\section{On the core interpretability logic \ilal}
Apart from the principles mentioned earlier in this paper the literature has considered various other principles too. Some of those are
\begin{enumerate}
%\item[${\sf M}$] 
%$A \rhd B \rightarrow A \wedge \Box C \rhd B \wedge \Box C$
%
%\item[${\sf P}$]
%$A \rhd B \rightarrow \Box (A \rhd B)$

%\item[${\sf M_0}$]
%$A \rhd B \rightarrow \Diamond A \wedge \Box C \rhd B \wedge \Box C$

\item[${\sf W}$:]
$A \rhd B \rightarrow A \rhd B \wedge \Box \neg A$

\item[${\sf W^*}$:]
$A \rhd B \rightarrow B \wedge \Box C \rhd B \wedge \Box C \wedge \Box \neg A$

\item[${\sf P_0}$:]
$A \rhd \Diamond B \rightarrow \Box (A \rhd B)$

\item[${\sf R}$:]
$A\rhd B \rightarrow \neg (A \rhd \neg C) \rhd B \wedge \Box C$
\end{enumerate}

In \cite{Visser:1988:preliminaryNotesOnInterpretabilityLogic}, \intl{All} was conjectured to be \ilw. In 
\cite{Visser:1991:FormalizationOfInterpretability} this conjecture was falsified and strengthened to 
a new conjecture, namely that \ilwstar, which is a 
proper extension of \ilw, is \intl{All}. In \cite{Joosten:1998:MasterThesis} it was proven that the logic \extil{W^*P_0} is a proper extension of
\ilwstar, and that \extil{W^*P_0} is a subsystem of \intl{All} (we write  \extil{W^*P_0} instead of \extil{\{ W^*,P_0\}}). This 
falsified the conjecture from \cite{Visser:1991:FormalizationOfInterpretability}.
In \cite{Joosten:1998:MasterThesis} it is also conjectured that \extil{W^*P_0} is not 
the same as \intl{All}.

In \cite{JoostenVisser:2000:IntLogicAll} it is conjectured that \extil{W^*P_0}=\intl{All} and this conjecture was refuted in \cite{GorisJoosten:2011:ANewPrinciple} by proving that the logic \extil{RW} is a subsystem of \intl{All} and a proper extension of \extil{W^*P_0}.\\
\medskip

It is easy to see that $A \rhd \Diamond B \to \Box (A\rhd \Diamond B) \ \in \ \ilp\cap \ilm$. In \cite{Visser:1997:OverviewIL} it was shown however that $A \rhd \Diamond B \to \Box (A\rhd \Diamond B) \ \notin \ \ilal$ thereby lowering the upper bound $\ilal \subseteq \ilp\cap \ilm$. Since $A \rhd \Diamond B \to \Box (A\rhd \Diamond B)$ is reminiscent of the modally incomplete principle $\principle {P_0}$, we remark here that the principle 
\[
A\rhd B \to \neg (A\rhd \Diamond C) \rhd B \wedge \Box \neg C
\]
implies $A \rhd \Diamond B \to \Box (A\rhd \Diamond B)$ so that it cannot be in \ilal either.

The current paper raises the previously known lower bound of \ilal. However, it seems unlikely that this will be the end of the story and the two series presented here seem amenable for interactions. Just by mere inspection of the frame conditions we observe that 
\begin{align*}
\mathcal F_n &= \forall w, x, y, z\ (\mathcal B_0(w,x,y,z)\Rightarrow \mathcal G_n(x,y,z)),\\
\mathcal{F}^n &= \forall w, x, y, z\ (\mathcal B_n(w,x,y,z)\Rightarrow \mathcal G_0(x,y,z)).
\end{align*}
suggesting possible interactions. For example, a combination of ${\sf R^1}$ and ${\sf R_1}$ could yield
\[
A\rhd B\rightarrow (C\rhd  A)\wedge\Diamond\neg(C\rhd\neg D)\wedge(E\rhd\Diamond F)
	\rhd
		B\wedge\Box D\wedge(E\rhd F).
\]
We note that the two series presented in this paper only spoke of $S$ relations that were imposed by the frame conditions. This suggests that a new conjecture can be formulated.

Let $\mathfrak F$ be a class of \il-frames. By $\il[\mathfrak F]$ we shall denote the interpretability logic corresponding to this class. That is, 
\[
\il[\mathfrak F]\ := \ \{ A \mid \forall F \in \mathfrak F\, \forall^{\sf valuation} {V} \ \la F, V \ra \models A\}.
\] 
We now define the class of frames $\mathfrak{All}$ to be the set of frames where any $S$ relation that is implied both by the \ilm and the \ilp frame condition is present. To make this more precise, let $P$ denote the first-order frame condition of \principle P and let $M$ denote the first-order frame condition of \principle M. Let $F(x,y,z)$ denote any sentence --first or higher order-- in the language $\{ R, \{S_x\}_{x\in W} \}$. We write $\ilp \models F(x,y,z) \to yS_xz$ to denote that for any Veltman frame $\mathcal F$ for which $\mathcal F \models P$ we also have $\mathcal F \models F(x,y,z) \to yS_xz$. Likewise, we shall speak of $\ilm \models F(x,y,z) \to yS_xz$. With this notation, we define
\[
\begin{array}{lll}
\mathfrak{All} &\ := & \{ \mathcal F \mid \Big ( \ilp \models (F(x,y,z) \to yS_xz) \ \& \ \ilm \models (F(x,y,z) \to yS_xz) \Rightarrow \\
 & & \ \ \ \ \ \ \ \ \mathcal F \models (F(x,y,z) \to yS_xz) \Big)\}.
\end{array}
\]
The second author poses the new conjecture
\begin{conjecture}
$\ilal = \il[\mathfrak{All}]$.
\end{conjecture}
It is easy to formulate the conjecture where the antecedent $F(x,y,z)$ is replaced by a set of sentences rather than a single sentence yet it seems hard to imagine that this is needed. Note that the conjecture only speaks of principles related to imposed $S$ relations. For example, this will leave out a principle like $A\rhd B \to (\Diamond A \wedge \Box \Box C \rhd B\wedge \Box C)$ as formulated in \cite{JoostenVisser:2000:IntLogicAll}.

\bibliographystyle{plain}
\bibliography{References}

\end{document}